\pdfoutput=1
\documentclass[a4paper]{amsart}

\usepackage{amsmath,amstext,amssymb,mathrsfs,amscd,amsthm,indentfirst}
\usepackage{amsfonts}
\usepackage{mathtools}
\usepackage{stmaryrd}
\usepackage{booktabs}
\usepackage{verbatim}
\usepackage{enumerate}
\usepackage{subfigure}
\usepackage{graphicx,import}
\graphicspath{{Images/}}
\numberwithin{equation}{section}
\usepackage{tabularx}
\usepackage[shortlabels]{enumitem}
\usepackage{bm}
\usepackage{amsmath,tikz-cd}
\usepackage{dot2texi}
\usetikzlibrary{positioning}
\usepackage{adjustbox}
\usetikzlibrary{arrows.meta}
\usetikzlibrary{arrows,shapes,decorations,automata,backgrounds,petri}

\makeatletter
\let\origsection\section
\renewcommand\section{\@ifstar{\starsection}{\nostarsection}}

\newcommand\nostarsection[1]
{\sectionprelude\origsection{#1}\sectionpostlude}

\newcommand\starsection[1]
{\sectionprelude\origsection*{#1}\sectionpostlude}

\newcommand\sectionprelude{%
  \vspace{1em}
}

\newcommand\sectionpostlude{%
  \vspace{1em}
}
\makeatother

\makeatletter
\let\origsubsection\subsection
\renewcommand\subsection{\@ifstar{\starsubsection}{\nostarsubsection}}

\newcommand\nostarsubsection[1]
{\subsectionprelude\origsubsection{#1}}

\newcommand\starsubsection[1]
{\subsectionprelude\origsubsection*{#1}}

\newcommand\subsectionprelude{%
  \vspace{0.8em}
}

\makeatother

\newcommand{\mm}{\mathbb}

\newcommand{\ima}{\mathrm{\small Im }\,}
\newcommand{\diff}{\mathrm{Diff}}
\newcommand{\emb}{\mathrm{Emb}}
\newcommand{\Id}{\mathrm{Id}}
\newcommand{\mmm}{\mathcal}

\newcommand{\Bun}{\mathrm{Bun}}
\newcommand{\Emb}{\mathrm{Emb}}
\newcommand{\Hom}{\mathrm{Hom}}
\newcommand{\iso}{\mathrm{Iso}}
\newcommand{\hAut}{\mathrm{hAut}}

\newcommand{\GL}{\mathrm{GL}}
\renewcommand{\O}{\mathrm{O}}
\newcommand{\U}{\mathrm{U}}
\newcommand{\SO}{\mathrm{SO}}
\newcommand{\Spin}{\mathrm{Spin}}
\newcommand{\Pin}{\mathrm{Pin}}
\newcommand{\timesover}[1]{{\raisebox{.5ex}{$\scriptstyle \bigtimes\limits_{#1}$}}}
\renewcommand{\:}{\mathop:}
\newcommand{\Fr}{\mathrm{Fr}}
\newcommand{\fr}{\mathrm{fr}}

\newcommand{\Map}{\mathrm{Map}}

\renewcommand{\int}{\mathrm{int}}

\newcommand{\hdrange}{\frac{g-3}{2}}
\newcommand{\confks}{\bm{C}_{kS^1}}
\newcommand{\conf}{\bm{C}}

\newcommand{\parallelsum}{\mathbin{\!/\mkern-5mu/\!}}
\newcommand{\dcup}[1]{{\raisebox{.5ex}{$\scriptstyle \coprod\limits_{#1}$}}}
\newcommand{\wellbehaved}{decorated-chiral}

\newcommand{\circlepathcomponent}{\confks(\mm{R}^\infty;(L(\Theta)\parallelsum L_\diamond(\GL_{d-1}))_0)}

\mathchardef\ordinarycolon\mathcode`\:
\mathcode`\:=\string"8000
\begingroup \catcode`\:=\active
  \gdef:{\mathrel{\mathop\ordinarycolon}}
\endgroup

\usepackage[all]{xy}
\CompileMatrices

\usepackage{amsthm}
\theoremstyle{plain}
\newtheorem{MainThm}{Theorem}

\newtheorem{theorem}{Theorem}[section]

\newtheorem{proposition}[theorem]{Proposition}
\newtheorem{lemma}[theorem]{Lemma}

\newtheorem{corollary}[theorem]{Corollary}

\theoremstyle{definition}
\newtheorem{definition}[theorem]{Definition}

\newtheorem{notation}[theorem]{Notation}
\newtheorem{example}[theorem]{Example}
\newtheorem{examples}[theorem]{Examples}

\theoremstyle{remark}

\newtheorem{remark}[theorem]{Remark}
\newtheorem*{remark*}{Remark}

\title[Decoupling decorations on moduli spaces]{Decoupling decorations on moduli spaces of manifolds}

\author{Luciana Basualdo Bonatto}
\email{luciana.bonatto@maths.ox.ac.uk}
\address{Mathematical Institute\\
Andrew Wiles Building\\
Oxford OX2 6GG \\
UK}

\date{\today}

\begin{document}
\newpage

\begin{abstract}
We consider moduli spaces of $d$-dimensional manifolds with embedded particles and discs. In this moduli space, the location of the particles and discs is constrained by the $d$-dimensional manifold. We will compare this moduli space with the moduli space of $d$-dimensional manifolds in which the location of such decorations is no longer constrained, i.e. the decorations are decoupled. We generalise work by B\"odigheimer--Tillmann for oriented surfaces and obtain new results for surfaces with different tangential structures as well as to higher dimensional manifolds. We also provide a generalisation of this result to moduli spaces with more general submanifold decorations and specialise in the case of decorations being unparametrised unlinked circles. 
\end{abstract}
\maketitle

\vspace{-0.5ex}\section{Introduction}

\noindent The diffeomorphism group of a smooth manifold and its classifying space are fundamental objects in topology. In particular, for a closed smooth manifold $W$, the space $B\diff(W)$ classifies the smooth fibre bundles with fibre $W$. When $W$ is a manifold with non-empty boundary, we consider $\diff(W)$ to be the group of those diffeomorphisms which are the identity near $\partial W$. The classifying spaces of such groups are also extremely important, as they are crucial for instance to the construction of topologically enriched categories of bordisms. To completely understand the classifying space of a diffeomorphism group is extremely difficult and such a description is only available for very few manifolds. One key strategy when studying $B\diff(W)$ is to understand how its homology behaves when changing the manifold $W$ by operations such as connected sum or gluing of cobordisms. In this paper, we use this strategy to study the stable homology of the decorated diffeomorphism group.

A $d$-dimensional manifold $W$ is said to be decorated if it is equipped with disjoint embeddings of k points and m discs $D^d$. The decorated diffeomorphism group of $W$, denoted $\diff^k_m(W)$ consists of those $\phi:W\to W$ which preserve the marked points and parametrized discs up to permutations. The classifying space $B\diff^k_m(W)$ has been studied from many different perspectives, for instance, considering the behaviour after increasing the number of marked points or discs (see \cite{MR3432333}).

For the case of $W=S_{g,b}$ the orientable surface of genus $g$ and $b$ boundary components, B\"odigheimer and Tillmann \cite{MR1851247} studied the comparison between $B\diff^k_m(S_{g,b})$ and $B\diff(S_{g,b})$. They used a decoupling map
    \begin{equation}\label{eq:decoupling-map-for-surfaces}
        d:\begin{tikzcd}
    B\diff^{+,k}_m(S_{g,b})\ar[rr, "f\times e_m \times e^k"] && B\diff^+(S_{g,b})\times B\Sigma_m\times B(\Sigma_k\wr\SO(2))
    \end{tikzcd}
    \end{equation}
where the map $f$ is induced by the inclusion $\diff^k_m(S_{g,b})\to \diff(S_{g,b})$, the map $e_m$ is induced by $\diff^k_m(S_{g,b})\to \Sigma_m$ recording the permutation of the marked discs, and $e^k$ is induced by the map $\diff^k_m(S_{g,b})\to \Sigma_k\wr \SO(2)$ recording the permutation of the marked points together with the induced map on their tangent space (see Figure~\ref{fig:decoupling} for a geometric representation of the decoupling map). B\"odigheimer and Tillmann showed that $d$ induces a homology isomorphism in degrees $\leq \frac{g}{3}$, therefore, in this range, we say that the decorations, which were bound to the manifold, get decoupled. The proof of this result relies strongly on Harer' stability theorem. Later, generalisations of Harer's result for non-orientable surfaces \cite{MR2367024} allowed Hanbury \cite{MR2439464} to generalise the decoupling result to such surfaces as well. In this paper, we further generalise this result to moduli spaces of manifolds in higher dimensions with tangential structures.

\begin{figure}[h!t]
    \centering\def\svgwidth{\columnwidth}
    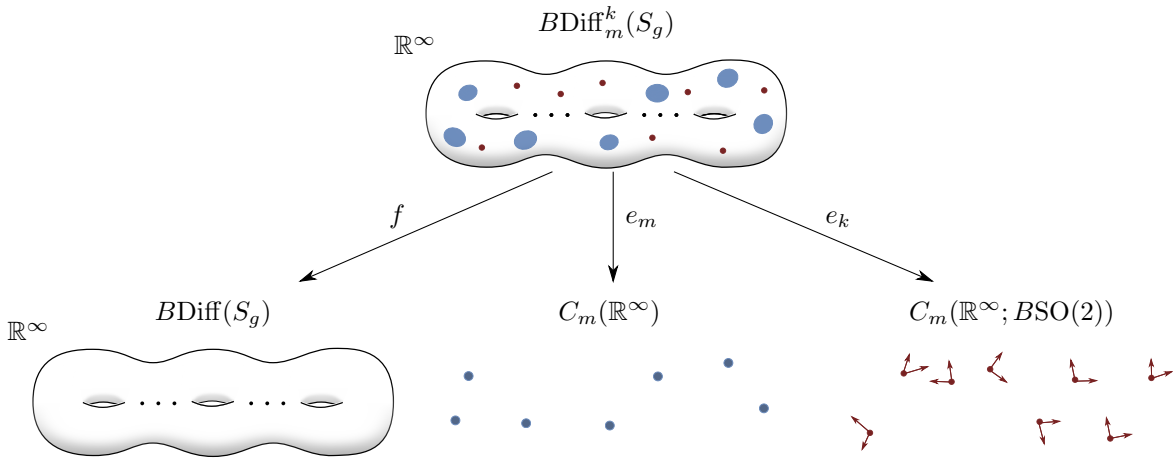
    \caption{Geometric representation of the decoupling map for the oriented moduli space of a surface $S_g$ as the product of three maps: forget the decorations, record the centre of the $m$ marked discs, and record the $k$ marked points and their oriented tangent spaces. For more details on the geometric interpretation see Section \ref{sec:geometric-view}.}
    \label{fig:decoupling}
    \end{figure}

\subsection{Tangential Structures}

Orientation, framings, spin structures and maps to a background space are all examples of a more general type of structure which can be described just from data on the tangent bundle of a manifold. A tangential structure is a topological space $\Theta$ equipped with a continuous $\GL_d$ action. A $\Theta$-structure on a $d$-dimensional manifold $W$ is a $\GL_d$-equivariant map $\rho_W:\Fr(TW)\to \Theta$, where $\Fr(TW)$ denotes the space of framings of the manifold $W$. A canonical example is $\Theta^{or}=\{\pm1\}$ with action given by multiplication with the sign of the determinant, and it is simple to see that a $\Theta^{or}$-structure on a manifold is a choice of orientation.

The space of all $\Theta$-structures has an action of $\diff(W)$ given by precomposition with the differential. Given a closed compact connected smooth manifold $W$ equipped with a $\Theta$-structure $\rho_W$, the moduli space $\mmm{M}^\Theta(W,\rho_W)$ of $W$ with $\Theta$-structures concordant to $\rho_W$ is defined as the path component of $\rho_W$ in the Borel construction
    \[\{\mbox{$\GL_d$-equivariant maps }\rho:\Fr(TW)\to \Theta\}\parallelsum \diff(W).\]
Examples of this construction are the classifying spaces $B\diff(W)$ and, when $W$ is orientable, $B\diff^+(W)$.

Analogously, the decorated moduli space of $(W,\rho_W)$, denoted $\mmm{M}^{\Theta,k}_m(W,\rho_W)$, is defined as the path component of $\rho_W$ in the Borel construction
    \[\{\mbox{$\GL_d$-equivariant maps }\rho:\Fr(TW)\to \Theta\}\parallelsum \diff^k_m(W).\]
    
If $W$ is a manifold with non-empty boundary, the moduli spaces $\mmm{M}^{\Theta}(W,\rho_W)$ and $\mmm{M}^{\Theta,k}_m(W,\rho_W)$ are defined analogously but only considering the $\GL_d$-equivariant maps $\rho:\Fr(TW)\to \Theta$ which agree with $\rho_W$ on $\Fr(TW_{|\partial W})$.
    
In this paper, we construct a decoupling map analogous to \eqref{eq:decoupling-map-for-surfaces}. For  $W$ a $d$-dimensional oriented manifold with non-empty boundary, this map is given by
    \[D:\begin{tikzcd}
    \mmm{M}^{\Theta,k}_m(W,\rho_W)\ar[rr, "F\times E_m \times E^k"] && \mmm{M}^\Theta(W,\rho_W)\times \Theta^m\parallelsum\Sigma_m\times (\Theta\parallelsum\GL_d^+)^k\parallelsum\Sigma_k
    \end{tikzcd}\]
The image of $D$ is a path-component of the codomain, which we denote 
    \[\mmm{M}^\Theta(W,\rho_W)\times \Theta^m_0\parallelsum\Sigma_m\times (\Theta\parallelsum\GL_d^+)^k_0\parallelsum\Sigma_k\]
and we show that, when a stabilisation condition is satisfied, the decoupling induces a homology isomorphism onto its image, in a range depending on the genus of $W$.

\subsection{Homology Stability}

A key ingredient in B\"odigheimer and Tillmann's proof for oriented surfaces is Harer's stability theorem \cite{MR786348}. It states that the map given by extending a diffeomorphism by the identity  
    \[\diff(S_{g,b+1})\to \diff(S_{g,b})\]
induces a map on classifying spaces which is homology isomorphism in degrees $\leq \frac{2}{3}g$ (the original bound by Harer was of $\frac{1}{3}g$, and the most recent bound is due to Randal-Williams \cite{randal2016resolutions}). Likewise, the more general decoupling result will depend on an analogous result for moduli spaces of manifolds with tangential structures: let $W$ be a compact connected $d$-dimensional manifold, and let $\rho_W$ be a fixed $\Theta$-structure on $W$. The manifold $W\setminus \int (D^d)$ is naturally endowed with a $\Theta$-structure $\rho_W'$ given by the restriction of $\rho_W$, and we have a map
    \begin{equation}\label{eq:homology-stability-map}
        \mmm{M}^\Theta(W\setminus\int (D^d),\rho_W')\to \mmm{M}^\Theta(W,\rho_W)
    \end{equation}
induced by extending a $\Theta$-structure by $\rho_W|_{_{D^d}}$ and a diffeomorphism by the identity. As for oriented surfaces, this map has been shown in many cases to induce a homology isomorphism in a range depending on the genus of $W$, for instance, this holds for surfaces with spin structures and framings. In dimensions $2n\geq 6$, this was shown to hold whenever $\rho_W$ is $n$-connected \cite[Corollary 1.7]{MR3665002}.

\subsection{Main results}

Throughout this paper, let $W$ be a compact, connected manifold of dimension $d\geq 2$.

\begin{MainThm}\label{main-decoupling}
    Let $W$ be an orientable manifold with non-empty boundary and $\rho_W$ be a fixed $\Theta$-structure on $W$. If the map \eqref{eq:homology-stability-map} induces a homology isomorphism in degrees $i\leq \alpha$ then the decoupling map 
    \[D:\mmm{M}^{\Theta,k}_{m}(W,\rho_W) \longrightarrow \mmm{M}^\Theta(W,\rho_W)\times \Theta^m_0\parallelsum\Sigma_m \times (\Theta\parallelsum\GL_d^+)_0^k\parallelsum\Sigma_k\]
    induces homology isomorphisms in degrees $i\leq \alpha$.
\end{MainThm}

In particular, the theorem above gives us new results on moduli spaces of surfaces with tangential structures, generalising the decoupling result of  B\"odigheimer and Tillmann to surfaces with spin structures, maps to a background space, framings, amongst others.

The main corollary of Theorem \ref{main-decoupling} is obtained in the context of manifolds of high even dimension, where the assumption on the map \eqref{eq:homology-stability-map} has been shown to hold whenever $\rho_W$ is $n$-connected. 

Although this connectivity assumption is quite restrictive, it is still possible to obtain further results for more general tangential structures using the techniques of \cite[Section 9]{MR3665002}. In particular, we prove the following:

\begin{MainThm}\label{main-bdiffWg1}
    Let $W_{g,1}=\#_g S^n\times S^n\setminus D^{2n}$, for $2n\geq 6$. Then for all $i\leq \frac{g-4}{3}$
    \[H_i(B\diff^{+,k}_m(W_{g,1})) \cong   H_i(B\diff^+(W_{g,1})\times \SO[0,n-1]^m\parallelsum\Sigma_m\times B\SO(2n)\langle n\rangle^k\parallelsum\Sigma_k)\]
    where $\SO[0,n-1]$ is the $n$-truncation of $\SO$ and $B\SO(2n)\langle n\rangle$ is the $n$-connected cover of $B\SO(2n)$.
\end{MainThm}

We also provide a computation of the cohomology of $B\diff^{+,k}_m(W_{g,1})$ using Theorem~\ref{main-bdiffWg1}.

\subsection{More general decorations}

When studying surfaces, it is natural to look at decorations by marked points and discs, however for high dimensional manifolds, one is allowed to explore more general types of decorations. This has been studied for instance in the recent work \cite{palmer2012configuration, kupers2013homological, palmer2018homological,palmer2018homological2}. We generalise the definition of the decorated moduli space of a manifold with more general decorations, we define a decoupling map and show that under a homology stability hypothesis the decoupling map induces a homology isomorphism in a range. We analyse more closely the case that the decorations are unlinked circles, because of its relation to the literature and also its relevance for string theory. The moduli space of a manifold $W$ with $k$ embedded circles and $\Theta$-structure is denoted $\mmm{M}^{\Theta}_{kS^1}(W,\rho_W)$ and is defined to be the moduli space of manifolds diffeomorphic to $W$, equipped with a $\Theta$-structure and $k$ marked unparametrised circles. As in Theorem \ref{main-decoupling}, we also get a splitting result, now in terms of the space of configurations of circles in $\mm{R}^\infty$ with labels on a space $\mmm{L}$ (see Definition~\ref{def: config of circles}), which we denote $\confks(\mm{R}^\infty,\mmm{L})$.

\begin{MainThm}\label{main-circle-decorations}
   Let $W$ be a simply-connected spin manifold of dimension $2n\geq 6$ with non-empty boundary, equipped  $k$ marked unparametrised circles and with a $\Theta$-structure $\rho_W:\Fr(TW)\to \Theta$ which is $n$-connected and such that $\Theta\parallelsum\GL_d$ is simply-connected. Then for all $i\leq \frac{g-4}{3}$
    \[H_i(\mmm{M}^{\Theta}_{kS^1}(W,\rho_W)) \cong H_i(\mmm{M}^{\Theta}(W,\rho_W)\times \confks(\mm{R^\infty};(L\Theta\parallelsum L_{null}\GL_{d-1}^+)_0))\]
    where $L-$ is the free loop space, $L_{null}$ is the subspace of nullhomotopic loops, and $(-)_0$ indicates a path-component that is specified in the proof.
\end{MainThm}

\subsection{Geometric Interpretation}\label{sec:geometric-view}

The spaces and maps used in the decoupling result all have a very concrete geometrical interpretation, which we briefly introduce. 

Let $\Theta^*=*$ be the point with the trivial $\GL_d$-action, then any manifold $W$ admits a unique $\Theta^*$-structure $\rho_W$. By definition, the moduli space $\mmm{M}^{\Theta^*}(W,\rho_W)$ is equivalent to $B\diff(W)$. So we can obtain a geometric interpretation for this moduli space from the specific model of $B\diff(W)$ given by the quotient $\emb(W,\mm{R}^\infty)/\diff(W)$. Then $\mmm{M}^{\Theta^*}(W)$ can be seen as the subspace of all submanifolds of $\mm{R}^\infty$ that are abstractly diffeomorphic to $W$. Analogously, fixing an arbitrary $\Theta$-structure, $\rho_W$, on a manifold $W$ (for instance a choice of orientation), the moduli space $\mmm{M}^\Theta(W,\rho_W)$ has a model as the space of all submanifolds of $\mm{R}^\infty$ that are diffeomorphic to $W$ together with a choice of $\Theta$-structure concordant to $\rho_W$. A detailed description of this model can be found in sections 6 and 7 of \cite{MR3718454}.

Through this perspective, the decorated moduli space $\mmm{M}^{\Theta,k}_m(W,\rho_W)$ is the space of all submanifolds of $\mm{R}^\infty$ that are diffeomorphic to $W$ together with $k$ marked points, $m$ marked parametrized discs, and a choice of $\Theta$-structure concordant to $\rho_W$.

Moreover, note that the spaces $\Theta^m_0\parallelsum\Sigma_m$ and $(\Theta\parallelsum\GL_d)_0^k\parallelsum\Sigma_k$ also have geometric models in terms of unordered configuration spaces with labels, which we denote respectively by $C_m(\mm{R}^\infty,\Theta_0)$ and $C_k(\mm{R}^\infty,(\Theta\parallelsum\GL_d)_0)$.

The decoupling map in Theorem \ref{main-decoupling}, is the product of three maps:
$$\mmm{M}^{\Theta,k}_m(W,\rho_W)\rightarrow \mmm{M}^{\Theta}(W,\rho_W)$$
is the map that simply forgets the decorations; the map 
$$\mmm{M}^{\Theta,k}_m(W,\rho_W)\rightarrow C_m(\mm{R}^\infty,\Theta_0)$$
records the centre of the marked discs together with local tangential structure information; and finally 
$$\mmm{M}^{\Theta,k}_m(W,\rho_W)\rightarrow C_k(\mm{R}^\infty,(\Theta\parallelsum\GL_d)_0)$$
is the map that records the positions of the marked points in $\mm{R}^\infty$ together with their tangent spaces and tangential structure information. See Figure \ref{fig:decoupling} for an illustration of these maps.

With this geometric interpretation, the decoupling result tells us that the homology of the space of decorated submanifolds of $\mm{R}^\infty$ of a fixed diffeomorphism type, in which the points and discs are constrained to our manifolds, can be understood in terms of the homology of a space where these points and discs are not constrained anymore, ie. they are decoupled.

\subsection{Outline of the paper}

Section \ref{section:preliminaries} recalls the basic concepts and results needed throughout the paper. We start by defining and giving examples of tangential structures and the topological moduli space of manifolds. Further, we prove auxiliary results on fibre sequences of Borel constructions and a spectral sequence argument which will be needed throughout the paper.

We define the decorated moduli space of manifolds and the decoupling map, and we prove Theorem \ref{main-decoupling} in Section \ref{section:proof-of-decoupling} . Subsequently, in Section~\ref{section:corollaries}, we prove the corollaries of Theorem \ref{main-decoupling} which provide decoupling results for surfaces with many tangential structures, as well as for manifolds of dimension $2n\geq 6$ with well-behaved tangential structures.

In Section \ref{section:decoupling-submanifolds} we define the generalisation of the decorated moduli space for more general types of submanifold decorations, define the decoupling map and prove the decoupling theorem in this case. We look closer at the case when decorations are embbedded unlinked circles and prove Theorem \ref{main-circle-decorations}.

Finally, in Section \ref{section:general-tangential-structures} we look at high dimensional manifolds with tangential structures that fail the hypothesis of Theorem \ref{main-decoupling} and provide a generalisation of the result for these cases. In particular, by applying this result to the manifolds $W_{g,1}$, we show Theorem \ref{main-bdiffWg1}.

\subsection*{Acknowledgements} I would like to thank my supervisor Ulrike Tillmann for her support and for the many useful conversations. In addition, I would like to thank Jan Steinebrunner and Tomáš Zeman for all the helpful discussions and comments. This work was carried out with the support of CNPq (201780/2017-8) and produced while the author was in residence at the Mathematical Sciences Research Institute in Berkeley, California, during the Spring semester of 2020.

\section{Preliminaries}\label{section:preliminaries}

\noindent In this section we recall the definition of tangential structures and give the examples that will be used in Section \ref{section:corollaries}. We also recall the moduli space of manifolds with tangential structures. We  also recall basic results about descending fibre sequences to homotopy quotients as well as a spectral sequence argument that will be used throughout the paper.

\subsection{Tangential structures}\label{sec:tangential-structures}
Throughout, we consider $W$ to be a smooth compact connected $d$-dimensional manifold, possibly with non-empty boundary. If $W$ is a closed manifold, we denote by $\diff(W)$ the group of diffeomorphisms of $W$ with Whitney $C^\infty$ topology. If $W$ has non-empty boundary, we assume it to be equipped with a collar neighbourhood of $\partial W$ and we denote by $\diff(W)$ the group of diffeomorphisms of $W$ which restrict to the identity on this collar. If moreover $W$ is an orientable manifold, we denote by $\diff^+(W)$ the subgroup of orientation preserving diffeomorphisms. Note that if $W$ has a non-empty boundary, any element of $\diff(W)$ is automatically orientation preserving.

Given a vector bundle $p:E\to B$, the frame bundle of $E$ over $B$ will be denote by $\Fr(E)$. Recall that the fiber of $\Fr(E)\to B$ over a fixed $b$ is the space of ordered bases of $p^{-1}(b)$, and this forms a $\GL_d$-principal bundle, with the action $A\cdot (v_1,\dots,v_n)=(A(v_1),\dots,A(v_n))$, for $A\in \GL_d$ and $(v_1,\dots,v_n)$ an ordered basis of $p^{-1}(b)$. Throughout this paper, we denote by $TW$ the tangent bundle of the manifold $W$, and by $\varepsilon^n\to B$ the trivial $n$-dimensional vector bundle over some base space $B$. In this paper, we will consider only real vector bundles. 

We now define tangential structures following the terminology established by Galatius and Randal-Williams in \cite{galatius2018moduli}.

\begin{definition}
    A \emph{tangential structure} for $d$-dimensional manifolds is a space $\Theta$ with a continuous action of $\GL_d\coloneqq\GL_d(\mm{R})$. A $\Theta$-structure on a $d$-manifold $W$ is a $\GL_d$-equivariant map $\rho:\Fr(TW)\to \Theta$. 
\end{definition}

Manifolds are usually equipped with data that can be described using tangential structures:

\begin{examples}\label{ex:tangential-structures} Let $W$ be a connected manifold.
    \begin{enumerate}[(a)]
        \item\label{itm:orientation-theta} An orientation consists of a coherent choice of which oriented bases of the tangent spaces are considered positive. Namely, this is the data of a $\GL_d$-equivariant map $\Fr(TW)\to \{\pm1\}$, where the action on $\Theta^{or}:=\{\pm1\}$ is given by multiplication by the sign of the determinant. Therefore, a $\Theta^{or}$-structure on a manifold is equivalent to a choice of orientation.
        
        \item\label{itm:trivial-theta} If we want to consider all manifolds with no extra data, we can use the trivial tangential structure $\Theta^*=\{*\}$ with the trivial action. Any manifold admits a unique $\Theta^*$-structure, and therefore it encodes no extra data. 
        
        \item\label{itm:framimings-theta} Framings on a manifold also are a tangential structure described by $\Theta^\fr=\GL_d$, since the data of a framing is precisely a continuous choice of basis for the tangent space at each point, which can be expressed as a $\GL_d$-equivariant map $\Fr(TW)\to \GL_d$.
        
        \item\label{itm:background-theta} Given a space $X$, we define the tangential structure of maps to $X$, by taking $\Theta_X=X$ with the trivial action of $\GL_d$. Then a $\Theta_X$-structure on a manifold $W$ is the data of a continuous a map $W\to X$.
    \end{enumerate}
\end{examples}

\begin{remark}\label{rmk:tangential-structures}
Many authors approach tangential structures in a different way, namely by defining it as a fibration $\theta:B\to B\O(d)$, and by setting a $\theta$-structure on a manifold $W$ to be a map $W\to B$ lifting the map $W\to B\O(d)$ which classified $TW$. There is a clear way of exchanging the two approaches using the correspondence between spaces with a $\GL_d$ action and spaces over $B\GL_d\simeq B\O(d)$, made through the principal $\GL_d$-bundle $E\GL_d\to B\GL_d$. Both the spaces $\Theta$ and $B$ associated to a given tangential structure will come into the decoupling result, so it is worth making precise the relation between them: given a fibration $\theta$, the pullback space 
\[\Theta:= E\GL_d\timesover{B{\tiny\GL}_d} B\]
is naturally equipped with a $\GL_d$ action. On the other hand, given a $\GL_d$-space $\Theta$, we can define $B$ as the Borel construction $\Theta\parallelsum \GL_d$ (ie. the quotient of $E\GL_d\times \Theta$ by the diagonal action of $\GL_d$). Then $E\GL_d\times \Theta\to B$ is a principal $\GL_d$-bundle, which means $B$ comes equipped with a map $\theta:B\to B\GL_d$. Since $E\GL_d$ is contractible, these processes are inverse up to equivariant fibre-wise weak equivalence.
\end{remark}

\begin{example}\label{ex:spin}
    Spin structures on an $n$-dimensional manifold are known to be classified by lifts along the fibration $\theta_{\Spin}:B\Spin\to B\O(d)\simeq B\GL_d$. So the corresponding $\Theta^\Spin$ is the pullback space
    \[E\GL_d\timesover{B{\tiny\GL}_d} B\Spin\]
    and it fits into the following diagram of fibre sequences
    \[\begin{tikzcd}
        \{\pm 1\}\times B\mm{Z}/2\ar[r] \ar[d, equal] & \Theta^\Spin \ar[r] \ar[d] \arrow[dr,phantom, "\lrcorner", very near start] & E\GL_d\ar[d]\\
        \{\pm 1\}\times B\mm{Z}/2\ar[r]               & B\Spin       \ar[r]        & B\GL_d
    \end{tikzcd}\]
    which implies that the space $\Theta^\Spin$ is homotopy equivalent to $\{\pm 1\}\times B\mm{Z}/2$. 
\end{example}

\begin{definition}
    Let $W$ be a closed manifold and $\Theta$ a fixed tangential structure. We define the \emph{space of $\Theta$-structures on $W$}, denoted $\Bun^\Theta(W)$, to be the space of all $\GL_d$-equivariant maps $\Fr(TW)\to \Theta$ equipped with the compact-open topology.

    Let $W$ be a manifold with non-empty boundary and a collar together with a $\GL_d$-equivariant map $\rho_{\partial}:\Fr(T\partial W\oplus \varepsilon)\to \Theta$. We define the \emph{space of $\Theta$-structures on $W$ restricting to $\rho_\partial$}, denoted $\Bun^\Theta_{\rho_\partial}(W)$, to be the space of all $\GL_d$-equivariant maps $\Fr(TW)\to \Theta$ that restrict to $\rho_\partial$ on $\partial W$.
\end{definition}

Given a manifold $W$ with non-empty boundary together with a $\GL_d$-equivariant map $\rho_{\partial}:\Fr(T\partial W\oplus \varepsilon)\to \Theta$, it is possible that the space $\Bun^\Theta_{\rho_\partial}(W)$ is empty, in the case where the chosen map $\rho_\partial$ cannot be extended to $\Fr(TW)$. An example, if $W$ is an orientable manifold with disconnected boundary and $\rho_\partial$ assigns non-compatible orientations for the different components of $\partial W$. These are not the cases we are interested in, and therefore throughout the paper, whenever $W$ is a manifold non-empty boundary we assume it comes equipped with a map $\rho_\partial$ which is the restriction of a $\Theta$-structure in $W$. 

\begin{examples}\label{ex:bun-spaces} Let $W$ be a manifold.
    \begin{enumerate}[(a)]
        \item If $\Theta^*$ is the trivial tangential structure of Example~\ref{ex:tangential-structures}\ref{itm:trivial-theta}, then there is only one $\Theta^*$-structure for any manifold, so $\Bun^{\Theta^*}(W)$ is a single point.
        
        \item\label{itm:name} Consider the tangential structure $\Theta^{or}$ for orientation described in Example~\ref{ex:tangential-structures}\ref{itm:orientation-theta}. If $W$ is a closed orientable manifold, it admits two $\Theta^{or}$-structures which implies that $\Bun^{\Theta^{or}}(W)$ consists of two points. On the other hand, if $W$ has a non-empty boundary and $\rho_\partial$ is a fixed $\Theta$ structure on $\partial W$, then $\Bun^{\Theta^{or}}_{\rho_\partial}(W)$ consists only of those $\GL_d$-equivariant maps $\Fr(TW)\to \Theta$ which restrict to $\rho_\partial$, and therefore consists of a single point.
        
        \item Given a space $X$, consider the tangential structure $\Theta_X$ defined in Example~\ref{ex:tangential-structures}\ref{itm:background-theta}. As discussed before, a $\Theta_X$ structure on a manifold $W$ is just a continuous map $W\to X$, and therefore, if $W$ is closed, $\Bun^{\Theta_X}(W)$ is the space of continuous maps from $W$ to $X$.
    \end{enumerate}
\end{examples}

\subsection{Moduli spaces of manifolds}\label{sec:recall-moduli spaces}

The action of the diffeomorphism group of $W$ on the tangent bundle $TW$ induces an action on the space $\Bun^\Theta(W)$ for any tangential structure $\Theta$. Explicitly, given $\phi\in\diff(W)$ and $\rho\in\Bun^\Theta(W)$,
    \[\phi\cdot \rho=\rho\circ D\phi^{-1}\]
where $D\phi:\Fr(TW)\to\Fr(TW)$ is the map induced by the differential of $\phi$.

\begin{definition}
    Let $W$ be a closed manifold and fix $\rho_W$ a $\Theta$-structure on $W$, we define $\Bun^\Theta(W,\rho_W)$ to be the orbit of the path-component of $\rho_W$ in $\Bun^\Theta(W)$ under the action of the diffeomorphism group $\diff(W)$. If $W$ has non-empty boundary $\Bun^\Theta(W,\rho_W)$ is defined to be the orbit of the path-component of $\rho_W$ in $\Bun^\Theta_{\rho_\partial}(W)$, where $\rho_\partial$ is the restriction of $\rho_W$ to the boundary.
    
    We define the \emph{moduli space of $W$ with $\Theta$-structures concordant to $\rho_W$} to be the Borel construction (ie. homotopy orbit space)
        \[\mmm{M}^\Theta(W,\rho_W):= \Bun^\Theta(W,\rho_W)\parallelsum \diff(W).\]
\end{definition}

\begin{remark}
     In the above definition, when $W$ is a manifold with boundary and $\rho_W$ a fixed $\Theta$-structure, we have omitted the symbol $\rho_\partial$ from the notation for the space $\Bun^{\Theta}(W,\rho_W)$. However, it should always be understood that there is a fixed boundary condition which is determined by  the restriction of the fixed $\rho_W$ to the boundary.
\end{remark}

The most important examples of these moduli spaces come from the simplest tangential structures: for the trivial tangential structure $\Theta^*$, the space $\Bun^{\Theta^*}(W)$ consists of a single point for any $W$ and therefore $\mmm{M}^{\Theta^*}(W,\rho_W)$ will be simply the classifying space $B\diff(W)$. On the other hand, if $W$ is an orientable manifold $\Bun^{\Theta^{or}}(W,\rho_W)$ consists of either one or two points depending on whether $\diff(W)$ has an element that reverses the orientation of $W$. In either case, the moduli space $\mmm{M}^{\Theta^{or}}(W,\rho_W)$ is homotopy equivalent to $B\diff^+(W)$.

\subsection{A lemma on fibre sequences and homotopy quotients}

In this section we prove a lemma that will be used throughout the paper to construct fibre sequences of moduli spaces from equivariant fibre sequences of diffeomorphism groups and spaces of $\Theta$-structures.

\begin{lemma}\label{2-out-of-3-for-fibrations}
    Given a commutative diagram 
        \[\begin{tikzcd}
            X \ar[rd, "f"] \ar[d, "h"] & \\
            Y \ar[r, "g"] & Z
        \end{tikzcd}\]
    such that $f$ and $h$ are Serre fibrations and $h$ is surjective, then $g$ is also a Serre fibration.
\end{lemma}

\begin{proof}
    We will show that $g$ has the homotopy lifting property with respect to any inclusion $D^i\times\{0\}\hookrightarrow D^{i}\times I$ using that both $f$ and $h$ have this property.
        \[\begin{tikzcd}
                &   X \ar[d, "h"] \ar[dd, bend left=49, "f"]\\
            D^{i}\times \{0\} \ar[r] \arrow[d, hook] & Y \ar[d, "g"] \\ 
            D^i\times I \ar[r] \ar[ru, dashed] & Z
        \end{tikzcd}\]
    
    Given a lift of $D^i\times\{0\}\to Y$ to $X$, we can construct a lift $\ell: D^i\times I\to X$ using that $f$ is a Serre fibration. Then $h\circ \ell:D^i\times I\to Y$ is a lift with respect to $g$. It only remains to see that any map $D^i\times\{0\}\to Y$ admits a lift to $X$, which we can prove by induction on $i$: for $i=0$, this is precisely the condition that $h$ is surjective, for $i>0$, this lift can be obtained using the identification $D^i\simeq D^{i-1}\times I$ and the fact that $h$ is a Serre fibration. 
\end{proof}

\begin{lemma}\label{fibration-descends-to-quotients} Let $G_i$ be a topological group and $p_i:M_i \to M_i/G_i$ be a $G_i$-principal bundle, for $i=1,2,3$.
    \begin{enumerate}[(a)]
        \item\label{itm:fibration-descends} If $\phi:G_2\to G_3$ is a continuous homomorphism and $f:M_2\to M_3$ is a $\phi$-equivariant fibration, then the induced map
            \[\psi:M_2/G_2 \rightarrow M_3/G_3\]
        is a fibration.
        \item\label{itm:fibre-over-descent} Given a short exact sequence
            \[0\to G_1 \to G_2 \to G_3 \to 0\]
         and a fibre sequence of equivariant maps
            $M_1 \to M_2 \to M_3$,
        the induced maps on quotients form a fibre sequence
            \[M_1/G_1 \rightarrow M_2/G_2 \rightarrow M_3/G_3\]
    \end{enumerate}
\end{lemma}

\begin{proof}
    \ref{itm:fibration-descends} By assumption, the map $p_2$ is a surjective fibration and the composition $\psi\circ p_2$ is equals the composition of fibrations $p_3\circ f$. Therefore, by Lemma~\ref{2-out-of-3-for-fibrations}, $\psi$ is a fibration.
    
    \ref{itm:fibre-over-descent} Diagrammatically, we want to show that given the diagram of fibre sequences below, there exists a fibre sequence fitting into the bottom row:
        \[\begin{tikzcd}[row sep=14pt]
            G_1 \ar[r, "\iota"] &   G_2 \ar[r, "\phi"]  &   G_3\\
            M_1 \ar[r, "i"] \ar[d, "p_1"] \arrow[loop above, out=120, in=70, distance=15] &   M_2 \ar[r,"f"] \ar[d, "p_2"] \arrow[loop above, out=120, in=70, distance=15] &   M_3 \ar[d, "p_3"] \arrow[loop above, out=120, in=70, distance=15]\\
            M_1/G_1 \ar[r, dotted]      &   M_2/G_2 \ar[r, dotted, "\psi"]      &   M_3/G_3 
        \end{tikzcd}\]
    By item \ref{itm:fibration-descends}, the map $\psi$ is a fibration, so all that remains is to identify its fibres. The composition $p_3\circ f$ is a fibration with fibre $G_2\cdot i(M_1)\subset M_2$. Then the fibre of $\psi$ is $p_2(G_2\cdot i(M_1))=p_2(i(M_1))$. Since the action of $G_3$ on $M_3$ is free, we know that for any $g\in G_2$ which is not in the kernel of $\phi$, the intersection $(g\cdot i(M_1))\cap i(M_1)$ is empty. So $p_2(i(M_1))$ is simply the quotient of $i(M_1)$ by the action of $\ker \phi=\iota(G_1)$. Then the map $M_1/G_1\to M_2/G_2$ is precisely the inclusion of the fibre of $\psi$.
\end{proof}

\begin{corollary}\label{fibration-descends-to-borel} Let $G_i$ be a topological group and $S_i$ be a $G_i$-space, for $i=1,2,3$.
    \begin{enumerate}[(a)]
        \item\label{itm:fibration-descends-b} If $\phi:G_2\to G_3$ is a continuous homomorphism and $f:S_2\to S_3$ is a $\phi$-equivariant fibration, then we can choose a model for the Borel constructions such that the induced map
            \[\psi:S_2\parallelsum G_2 \rightarrow S_3 \parallelsum G_3\]
        is a fibration.
        \item\label{itm:fibre-over-descent-b} Given a short exact sequence
            \[0\to G_1 \to G_2 \xrightarrow{\phi} G_3 \to 0\]
         such that $\phi$ is a principal bundle, and a fibre sequence of equivariant maps
            $S_1 \to S_2 \to S_3$,
        the induced maps on quotients form a homotopy fibre sequence
            \[S_1\parallelsum G_1 \rightarrow S_2\parallelsum G_2 \rightarrow S_3\parallelsum G_3\]
    \end{enumerate}
\end{corollary}

\begin{proof}
    Both statements follow from Lemma~\ref{fibration-descends-to-quotients} and the following observations: fix $EG_2$, then the inclusion $\iota:G_1\to G_2$ induces a $G_1$ action on $EG_2$ and since $\phi$ is a principal $G_1$-bundle, then so is $EG_2\to EG_2/G_1$. Therefore the space $EG_2$ is a model for $EG_1$ as well. Also, any model for the space $EG_3$ carries an action of $G_2$ via the map $\phi:G_2\to G_3$, and in particular, the quotient of $EG_2\times EG_3$ by $G_2$ is still a $G_2$-principal bundle. Then the proof follows directly from applying Lemma~\ref{fibration-descends-to-quotients} to the diagram
    \[\begin{tikzcd}[row sep=14pt]
            G_1 \ar[r, "\iota"] &   G_2 \ar[r, "\phi"]  &   G_3\\
            S_1\times EG_2 \ar[r] \ar[d] \arrow[loop above, out=120, in=70, distance=15] &   S_2\times EG_2\times EG_3 \ar[r] \ar[d] \arrow[loop above, out=120, in=70, distance=15] &   S_3\times EG_3 \ar[d] \arrow[loop above, out=120, in=70, distance=15]\\
            S_1\parallelsum G_1 \ar[r, dotted]      &   S_2\parallelsum G_2 \ar[r, dotted]      &  S_3\parallelsum G_3 
        \end{tikzcd}\]
    where the middle row is the product of the fibre sequence $S_1\to S_2\to S_3$ with the trivial fibre sequence
    \[\begin{tikzcd}[column sep=14pt]
            EG_2 \ar[r] & EG_2\times EG_3 \ar[r] & EG_3
        \end{tikzcd}\]
    and the action of the groups is the diagonal action. The commutativity of the diagram follows from the fact that the action of $G_1$ on $EG_3$ induced by $\phi\circ\iota$ is trivial.
\end{proof}

In particular, applying the above corollary to the trivial fibration $*\to *$, gives us the well-known result that a short exact sequence of groups $G_1\to G_2\to G_3$ induces a fibre sequence on classifying spaces 
        \[\begin{tikzcd}[column sep=14pt]
            BG_1 \ar[r] & BG_2 \ar[r] & BG_3
        \end{tikzcd}\]

\subsection{The spectral sequence argument}

The decoupling result will be deduced from the comparison of the homology spectral sequence associated to fibre sequences of moduli spaces. The result then follows from a well-known spectral sequence result which we recall (for a proof see \cite{MR3432333}):

\begin{lemma}[\textit{Spectral Sequence Argument}]\label{specseqargument}
    Let $f:E^\bullet_{p,q}\to \Tilde{E}^\bullet_{p,q}$ be a map of homological first quadrant spectral sequences. Assume that
    \begin{align*}
        f:E^2_{p,q}\xrightarrow{\cong} \Tilde{E}^2_{p,q} && \mbox{for $0\leq p<\infty$ and $0\leq q\leq l$}.
    \end{align*}
    Then $f$ induces an isomorphism on the abutments in degrees $*\leq l.$
\end{lemma}
\section{The Decoupling Theorem}\label{section:proof-of-decoupling}

\noindent In this section we introduce decorated manifolds, the decorated moduli space, and the maps that are in the centre of the decoupling theorem: the forgetful map and evaluation map. We end by defining the decoupling map and proving the decoupling theorem.

For now, we focus on decorations being points and discs. These are the extreme cases: the simplest embedded manifolds of lowest and highest possible dimension. In section \ref{section:decoupling-submanifolds} we show how this can be extended to more general submanifold decorations, focusing on the case of manifolds decorated with embedded unlinked circles.

\subsection{The decorated moduli space and the forgetful map}\label{sec:decorated-moduli-space}  Throughout this section, let $W$ be a compact connected smooth manifold. We will study manifolds equipped with decorations:

\begin{definition}\label{decorations}
    A $d$-dimensional \emph{manifold with decorations} consists of a manifold $W$ together with a set of distinct marked points in its interior $p_1, \dots, p_k\in W\setminus\partial W$ and disjoint embeddings $\phi_1,\dots,\phi_m:D^{d}\hookrightarrow W\setminus(\partial W\cup\{p_1,\dots,p_k\})$, with $k,m\in\mm{N}$. If $W$ is orientable, we require all embeddings to be oriented in the same way. We refer to these choices as \emph{decorations} on our manifold.

    Given a manifold $W$ with decorations, we define the \emph{decorated diffeomorphism group} $\diff^k_m(W)$ to be the subgroup of $\diff(W)$ of the diffeomorphisms $\psi$ such that  
        \begin{align*}
            \psi\circ \phi_j=&\phi_{\alpha(j)} & \psi(p_i)=&p_{\beta(i)} 
        \end{align*}
    for some $\alpha\in \Sigma_m$ and $\beta\in\Sigma_k$.
\end{definition}

In other words, we are looking at the diffeomorphisms that preserve the marked points and parametrized discs up to permutations. Note that the notation $\diff^k_m(W)$ does not record which points and embedded discs comprise the decorations. The following lemma justifies this notation. 

\begin{lemma}\label{lemma: diffkm independednt of choice}
    If $d\geq 2$, the isomorphism type of $\diff^k_m(W)$ does not depend on the choice of the $k$ points and $m$ embedded discs that comprise the decorations.
\end{lemma}

\begin{proof}
    For any two collections of decorations in $W$ denoted \[(p_1,\dots,p_k,\phi_1,\dots,\phi_m)\text{ and }(p_1',\dots,p_k',\phi_1',\dots,\phi_m'),\]
    there exists a diffeomorphism $\psi$ of $W$ such that
    \begin{align*}
        \psi(p_i)=p_i' & & \psi\circ\phi_i=\phi_i'
    \end{align*}
    which can be constructed recursively by extending isotopies of the points and discs to diffeotopies of $W$ as described in \cite[Chapter 8, Theorems 3.1, 3.2]{Hirsch}. Then conjugation with $\psi$ defines an isomorphism between the group of diffeomorphisms preserving $(p_1,\dots,p_k,\phi_1,\dots,\phi_m)$ and the one preserving $(p_1',\dots,p_k',\phi_1',\dots,\phi_m')$.
\end{proof}

We are now ready to define the analogue of the moduli space, including the decorations:

\begin{definition}
    Given a manifold $W$ with a $\Theta$-structure $\rho_W$, we define the \emph{decorated moduli space of $W$} with $k$ points and $m$ discs to be
    \begin{align*}
        \mmm{M}^{\Theta,k}_m(W,\rho_W):=\Bun^{\Theta}(W,\rho_W)\parallelsum \diff^k_m(W).
    \end{align*} 
\end{definition}

Recall that, if $W$ has non-empty boundary, then $\diff(W)$ consists only of those diffeomorphisms fixing a collar of the boundary and the elements of $\Bun^\Theta(W,\rho_W)$ agree with $\rho_W$ on $\partial W$.

We define the forgetful map 
    \begin{equation}\label{eq:forgetful-map}
        F:\mmm{M}^{\Theta,k}_m(W,\rho_W)\to \mmm{M}^\Theta(W,\rho_W)
    \end{equation}
to be the one induced by the identity map on $\Bun^\Theta(W)$ and the subgroup inclusion $\diff^k_m(W)\to \diff(W)$.

\subsection{The evaluation map}

Let $W$ be a decorated manifold with $k$ marked points and $m$ marked discs. For each marked point $p_i$, we choose once and for all a frame of $T_{p_i}W$, and if $W$ is oriented, we ask that these frames have the same orientation. We also fix throughout this section $N\subset W$ which is the union of a tubular neighbourhood of the marked points and the interiors of the parametrized discs. We denote by $W_{m+k}$ the manifold $W\setminus N$. The decoupling result follows from understanding the difference between the decorated moduli space of $W$ and the moduli space of $W_{m+k}$.

For instance, assume $k=0$ and $m=1$, then there is a group isomorphism
    \[\diff(W_1)\to \diff_1(W)\]
given by extending a diffeomorphism on $W_1$ by the identity on the marked disc. More generally, if $W$ is a manifold with $m$ embedded discs, the map
    $$e_m:\diff_m(W)\to \Sigma_m $$
taking a diffeomorphism $\phi$ to the $\alpha\in \Sigma_m$ recording the permutation induced on the discs by $\phi$, is a surjective homomorphism with kernel $\diff(W_m)$, where, as above, $W_m$ is the manifold obtained from $W$ by removing the interior of the $m$ embedded discs.

Assume now $W$ has $k$ marked points $\{p_1,\dots,p_k\}$ and no marked discs. We still get a homomorphism
    \[\diff(W_k)\to \diff^k(W)\]
by extending a diffeomorphism on $W_k$ by the identity on the removed neighbourhood of the points, but this is not an isomorphism, since the elements of $\diff^k(W)$ are not required to fix the entire neighbourhood of the marked points. A way to understand the diffeomorphisms around these is by looking at the differential map on the chosen frames at the marked points. So we define a map to the wreath product
    \begin{align*}
        e^k:\diff^k(W)&\longrightarrow \Sigma_k\wr \GL_d\\
        \phi &\longmapsto (D_{p_1}\phi,\dots,D_{p_k}\phi,\beta)
    \end{align*}
where $\beta\in \Sigma_k$ is the permutation induced on the marked points by $\phi$. The image of $e^k$ depends the manifold $W$.

\begin{definition}
    An orientable decorated manifold $W$ with $k$ marked points and $m$ marked discs is called \emph{\wellbehaved} if every $\phi \in \diff_m^k(W)$ preserves the orientation.
\end{definition}

\begin{remark}
    If the manifold $W$ is decorated by $m>0$ discs or $\partial W\neq \emptyset$ then it is immediately \wellbehaved. There are also many manifolds for which any diffeomorphisms (not necessarily decorated) are orientation preserving, these are called \emph{chiral} manifolds. A classical example is $\mm{C}P^2$, which can be deduced by analysing the automorphisms of its cohomology ring. Trivially, any chiral manifold is always decorated-chiral.
\end{remark}

It follows from \cite[Lemmas 2.3 and 2.4]{MR3432333} that:
 
\begin{lemma}[\cite{MR3432333}]\label{fibrationongroups}
    Let $W$ be a compact connected decorated manifold, then:
    \begin{enumerate}[(a)]
        \item\label{itm:e-is-fibration} the map 
            \[e:\begin{tikzcd}[column sep=1cm] \diff^k_{m}(W)\ar[r, "e_m\times e^k"] & \Sigma_m\times (\Sigma_k\wr \GL_d^\dagger) \end{tikzcd}\]
        is a surjective principal bundle, where the group $\GL_d^\dagger$ is $\GL_d^+$ if $W$ is \wellbehaved, and $\GL_d$ otherwise.
        \item\label{itm:homotopy-fibre-e} $\diff(W_{m+k})$ is the homotopy fibre of $e$.
    \end{enumerate}
\end{lemma}

\begin{remark}
    Identifying the image of $e$ is important because we want to use a Serre spectral sequence to compare the homology of the total spaces of two fibre sequences. Therefore it is important to identify precisely the images of the fibrations we define.
\end{remark}

A generalisation of the above lemma provides a fibre sequence on moduli spaces with tangential structures which is the key to the proof of the decoupling.

\begin{proposition}\label{thefibration}
    Let $W$ be a compact connected decorated manifold and $\rho_W$ a fixed $\Theta$-structure on $W$, then:
    \begin{enumerate}[(a)]
        \item\label{itm:E-is-fibration} The homomorphism $e$ induces an {evaluation map}
            \[E:\begin{tikzcd} \mmm{M}^{\Theta,k}_m(W,\rho_W) \ar[r] & \Theta^m\parallelsum\Sigma_m\times (\Theta\parallelsum \GL_d^\dagger)^k\parallelsum \Sigma_k \end{tikzcd}\]
        which is a Serre fibration onto the path component which it hits, where the group $\GL_d^\dagger$ is $\GL_d^+$ if $W$ is \wellbehaved, and $\GL_d$ otherwise.
        \item\label{itm:homotopy-fibre-E} Let $W_{m+k}$ be equipped with the $\Theta$-structure $\rho_{W_{m+k}}$ given by the restriction of $\rho_W$. Then 
            \[\mmm{M}^\Theta(W_{m+k},\rho_{W_{m+k}})\]
        is the homotopy fibre of $E$ over its image.
    \end{enumerate}
\end{proposition}

To prove the Proposition, we will need the following lemma:

\begin{lemma}\label{bun-fibration}
    Let $W$ be a connected manifold and $S$ a smooth submanifold, then
    the restriction map
            \[r_S:\begin{tikzcd} \Bun^\Theta(W)\to \Map_{\GL_d}(\Fr(TW|_S),\Theta) \end{tikzcd}\]
        is a Serre fibration. 
\end{lemma}

\begin{proof}
    For any $i\geq 0$, a lift for the diagram
        \[\begin{tikzcd}
            D^i\times \{0\} \ar[r] \ar[d] & \Bun^\Theta(W) \ar[d, "r_S"]\\
            D^i\times I \ar[r] \ar[ru, dotted] & \Map_{\GL_d}(\Fr(TW|_S),\Theta)
        \end{tikzcd}\]
    is equivalent to a $\GL_d$-equivariant extension of the following
        \begin{equation}\label{equivariant-lifting}
            \begin{tikzcd}
            (D^i\times \{0\}\times \Fr(TW))\cup (D^i\times I\times \Fr(TW|_S)) \ar[r, "\rho"] \ar[d, hook] & \Theta\\
            D^i\times I\times \Fr(TW)  \ar[ru, dotted] & 
            \end{tikzcd}
        \end{equation}
    
    Since the inclusion $S\hookrightarrow W$ is an embedding, there exists a strong deformation retract
        $$ r: D^i\times I\times W \longrightarrow (D^i\times \{0\}\times W)\cup (D^i\times I\times S).$$
    If $i$ denotes the inclusion of $(D^i\times \{0\}\times W)\cup (D^i\times I\times S)$ into $D^i\times I\times W$, we have an isomorphism 
        $$f:D^i\times I\times \Fr(TW)\xrightarrow{\cong} r^*i^*(D^i\times I\times \Fr(TW))$$
    which is the identity on $(D^i\times \{0\}\times \Fr(TW))\cup (D^i\times I\times \Fr(TW|_S))$.
    
    Therefore the composite
    \[\begin{tikzcd}
        D^i\times I\times \Fr(TW) \ar[r, "f"]& r^*i^*(D^i\times I\times \Fr(TW)) \ar[r, "r^*"]& i^*(D^i\times I\times \Fr(TW)) \ar[r, "\rho"] & \Theta
    \end{tikzcd}\]
    gives a lift to diagram \ref{equivariant-lifting}. This implies that the map $\Bun^\Theta(W)\to \Map_{\GL_d}(\Fr(TW|_S),\Theta)$ is a Serre fibration.
\end{proof}
\medskip

\begin{proof}[Proof of Proposition \ref{thefibration}]
    \ref{itm:E-is-fibration} Let $P\subset W$ be the union of the $k$ marked points and the centres of the $m$ marked discs. By Lemma \ref{bun-fibration}, the restriction map
        \[r_P:\Bun^\Theta(W,\rho_W)\to \Map_{\GL_d}(\Fr(TW|_P),\Theta)\]
    is a Serre fibration. For each marked point, we chose a frame of its tangent space. Each point in the centre of a marked disc, comes with a preferred frame induced by the parametrization of the disc. So every point in $P$ is equipped with a frame of its tangent space, and this gives us a diffeomorphism $\Fr(TW|_P)\cong \GL_d\times P$. Therefore the space of $\GL_d$-equivariant maps $\Fr(TW|_P)\to \Theta$ can be identified with the space of continuous maps $P\to \Theta$, which is just $\Theta^{m}\times \Theta^{k}$. 
    The result then follows by applying Corollary~\ref{fibration-descends-to-borel} to combine the fibration $r_P$ with the homomorphism of Lemma \ref{fibrationongroups}
        \begin{equation}
            \diff^k_m(W)\xrightarrow{e}  \Sigma_m\times(\Sigma_k\wr \GL_d^\dagger).
        \end{equation}
    We apply Corollary~\ref{fibration-descends-to-borel}\ref{itm:fibration-descends-b} by taking $G_2=\diff^k_m(W)$ and $S_2=\Bun^\Theta(W,\rho_W)$, with the usual action by precomposition with the differential. On the other hand, we take $G_3=\Sigma_m\times(\Sigma_k\wr \GL_d^\dagger)$, and $S_3=\Theta^{m}\times \Theta^{k}$ with the following action: the space $\Theta^{m}\times \Theta^{k}$ can be spit into the $m$ factors corresponding to the marked discs and $k$ factors corresponding to the marked points. Then we have an action of $\Sigma_m\times(\Sigma_k\wr \GL_d^\dagger)$ on $\Theta^{m}\times \Theta^{k}$ induced by the actions
        \[\begin{tikzcd}[column sep=tiny] \Sigma_m & \Theta^m \ar[loop left, out=195, in= 165, distance=10] && \Sigma_k\wr\GL_d^\dagger & \Theta^k. \ar[loop left, out=195, in= 165, distance=10] \end{tikzcd}\]
    Then the fibration $\Bun^\Theta(W,\rho_W)\to \Theta^{m}\times \Theta^{k}$ is $e$-equivariant and therefore, by Corollary~\ref{fibration-descends-to-borel}\ref{itm:fibration-descends-b}, we have a fibration
        \[E:\begin{tikzcd} \mmm{M}^{\Theta,k}_m(W,\rho_W) \ar[r] & \Theta^m\parallelsum\Sigma_m\times \Theta^k\parallelsum(\Sigma_k\wr \GL_d^\dagger). \end{tikzcd}\]
    Since $E\Sigma_k\times(E\GL_d)^k$ is a model for $E(\Sigma_k\wr \GL_d^\dagger)$, then
    \[(\Theta\parallelsum \GL_d^\dagger)^k\parallelsum \Sigma_k\]
     is a model for $\Theta^k\parallelsum(\Sigma_k\wr \GL_d^\dagger)$, and the result follows.
    
    \medskip
    \ref{itm:homotopy-fibre-E} Recall that $W_{m+k}$ is defined as the submanifold $W\setminus N$, where $N$ is the union of a tubular neighbourhood of the marked points and the interiors of the marked discs. The restriction $\rho_{W_{m+k}}$ of $\rho_W$ is a $\Theta$-structure on $W_{m+k}$. In the remainder of the proof, we will show that $\mmm{M}^\Theta(W_{m+k},\rho_{W_{m+k}})$ is the homotopy fibre of $E$. 
    
    A description of the fibre of $E$ can be obtained using Corollary~\ref{fibration-descends-to-borel}\ref{itm:fibre-over-descent-b} with the short exact sequence of groups being 
        \begin{equation}\label{eq:prop-ses}
            \ker e \to \diff^k_m(W)\xrightarrow{e}  \Sigma_m\times(\Sigma_k\wr \GL_d^\dagger)
        \end{equation}
    and the fibre sequence of $S_1\to S_2\to S_3$ being the one associated to the fibration $r_P$ of item \ref{itm:E-is-fibration}. The fibre of $r_P$ over $r_P(\rho_W)$ is the subspace of all elements of $\Bun^\Theta(W,\rho_W)$ which restrict to $r_P(\rho_W)$ over $P$, which we here denote $\Bun^\Theta_P(W,\rho_W)$. This space carries an action of $\ker e$ by precomposition with the differential, and it is simple to check that this fibre sequence is equivariant with respect to \eqref{eq:prop-ses}. Then by Corollary~\ref{fibration-descends-to-borel}\ref{itm:fibre-over-descent-b}, the fibre of the evaluation map $E$ is given by 
        \[\Bun^\Theta_P(W,\rho_W)\parallelsum\ker e.\]
    
    Applying Lemma \ref{bun-fibration} to both submanifolds $P$ and $N$, we obtain two fibrations fitting into the following commutative diagram
        \[\begin{tikzcd}
            \Bun^\Theta(W,\rho_W) \ar[r, "r_N"] \ar[d, "="] & \Map_{\GL_d}(\Fr(TW|_N),\Theta) \phantom{\cong \Theta^{m}\times \Theta^{k}} \ar[d, "i^*"]\\
            \Bun^\Theta(W,\rho_W) \ar[r, "r_P"] & \Map_{\GL_d}(\Fr(TW|_P),\Theta)\cong \Theta^{m}\times \Theta^{k} 
        \end{tikzcd}\]
    where the right-hand vertical map is induced by the inclusion $i:P\hookrightarrow N$. Since $\Fr(TD^d)$ is isomorphic to $\GL_d\times D^d$ as $\GL_d$-bundles, and the spaces $\GL_d\times D^d$ and $\GL_d\times \{*\}$ are homotopy equivalent as $\GL_d$-spaces, the map $i^*$ is a homotopy equivalence.
    In particular, this implies that the map from the fibre of $r_N$ to $\Bun^\Theta_P(W,\rho_W)$ is a homotopy equivalence. 
    
    The fibre of $r_{_N}$ over $r_{_N}(\rho_W)$ is by definition the space of all $\Theta$ structures on $W$ which agree with $\rho_W$ on $N$. We claim that this space is homeomorphic to $\Bun^\Theta(W_{m+k},\rho_{W_{m+k}})$, since the restriction map $r_{_{W_{m+k}}}$ takes the fibre of $r_{_N}$ bijectively to $\Bun^\Theta(W_{m+k},\rho_{W_{m+k}})$ and it has an inverse given by extending an element by $r_{_N}(\rho_W)$.
    
    So we have a commutative diagram of principal fibre bundles
        \[\begin{tikzcd}
            \diff(W_{m+k}) \ar[r, "\simeq"] \ar[d] & \ker e \ar[d]\\
            \Bun^\Theta(W_{m+k},\rho_{W_{m+k}})\times E\diff(W) \ar[r, "\simeq"] \ar[d] & \Bun^\Theta_P(W,\rho_W)\times E\diff(W) \ar[d]\\
            \mmm{M}^\Theta(W_{m+k},\rho_{W_{m+k}}) \ar[r] & \Bun^\Theta_P(W,\rho_W)\parallelsum \ker e
        \end{tikzcd}\]
    where the top horizontal map is a homotopy equivalence by Lemma \ref{fibrationongroups} and the middle map is a homotopy equivalence by the discussion above. Therefore the map 
        \[\mmm{M}^\Theta(W_{m+k},\rho_{W_{m+k}}) \to \Bun^\Theta_P(W,\rho_W)\parallelsum \ker e\]
    is also a homotopy equivalence, as required.
\end{proof}

\subsection{Proof of the Decoupling}

In this section we prove the decoupling result by comparing the homotopy fibration sequence from Proposition \ref{thefibration} to the product fibre sequence via the aforementioned spectral sequence argument.

\begin{definition}
    The \emph{decoupling map}
    \[D:\begin{tikzcd}
        \mmm{M}^{\Theta,k}_m(W,\rho_W) \ar[r, "F\times E"] & \mmm{M}^\Theta(W,\rho_W)\times \Theta^m\parallelsum\Sigma_m\times (\Theta\parallelsum \GL_d^\dagger)^k\parallelsum\Sigma_k
    \end{tikzcd}\]
    is the product of the forgetful map \eqref{eq:forgetful-map} and the evaluation map $E$ defined in Proposition~\ref{thefibration}.
\end{definition}

We now restate the decoupling theorem:

\begin{theorem}\label{decoupling}
    Let $W$ be a smooth connected compact manifold equipped with a $\Theta$-structure $\rho_W$. If the map $$\tau:H_i(\mmm{M}^\Theta(W_{m+k},\rho_{W_{m+k}}))\to H_i(\mmm{M}^\Theta(W,\rho_W))$$ induces a homology isomorphism in degrees $i\leq \alpha$, then for all such $i$ the decoupling map $D$ induces an isomorphism
    \[H_i(\mmm{M}^{\Theta,k}_m(W,\rho_W))\cong H_i(\mmm{M}^\Theta(W,\rho_W)\times \Theta^m_0\parallelsum\Sigma_m\times (\Theta\parallelsum \GL_d^\dagger)_0^k\parallelsum\Sigma_k)\]
    where $(-)_0$ denotes a path component of the image of $\rho_W$, and the group $\GL_d^\dagger$ is $\GL_d^+$ if $W$ is orientable and \wellbehaved, and $\GL_d$ otherwise.
\end{theorem}

\medskip
\begin{proof}
    By Proposition \ref{thefibration}, $E$ is a fibration. Since $\mmm{M}^{\Theta,k}_m(W,\rho_W)$ is connected by definition, we know that the image of $E$ is precisely the path component
        \[\Theta^m_0\parallelsum\Sigma_m\times (\Theta\parallelsum \GL_d^\dagger)_0^k\parallelsum\Sigma_k.\]
    So we have a homotopy fibre sequence
    \[\begin{tikzcd}
        {\mmm{M}^\Theta(W_{m+k},\rho_{W_{m+k}})} \arrow[r] & {\mmm{M}^{\Theta,k}_{m}(W,\rho_W)} \arrow[r, "E"] & \Theta^m_0\parallelsum{\Sigma_m}\times (\Theta\parallelsum \GL_d^\dagger)_0^k\parallelsum{\Sigma_k}.
    \end{tikzcd}\]
    The proof of the theorem follows from the comparison between this homotopy fibre sequence and the trivial fibre sequence associated to the projection map
    \[\begin{tikzcd}
        {\mmm{M}^\Theta(W,\rho_{W})\times \Theta^m_0\parallelsum{\Sigma_m}\times (\Theta\parallelsum \GL_d^\dagger)_0^k\parallelsum{\Sigma_k}} \arrow[r] & \Theta^m_0\parallelsum{\Sigma_m}\times (\Theta\parallelsum \GL_d^\dagger)_0^k\parallelsum{\Sigma_k}
    \end{tikzcd}\]
    By definition of the maps in Proposition~\ref{thefibration}, the following is a commutative diagram of homotopy fibre sequences 
    
    \medskip
    
    \[\begin{tikzcd}[column sep=small]
        {\mmm{M}^\Theta(W_{m+k},\rho_{W_{m+k}})} \arrow[r] \arrow[d, "\tau"] & {\mmm{M}^{\Theta,k}_{m}(W,\rho_W)} \arrow[r, "E"] \arrow[d, "D"] & \Theta^m_0\parallelsum_{\Sigma_m}\times (\Theta\parallelsum \GL_d^\dagger)_0^k\,\parallelsum_{\Sigma_k}    \arrow[d, equal] \\
        {\mmm{M}^\Theta(W,\rho_{W})} \arrow[r] & {\mmm{M}^\Theta(W,\rho_{W})\times \Theta^m_0\parallelsum_{\Sigma_m}\times (\Theta\parallelsum \GL_d^\dagger)_0^k\,\parallelsum_{\Sigma_k}} \arrow[r] & \Theta^m_0\parallelsum_{\Sigma_m}\times (\Theta\parallelsum \GL_d^\dagger)_0^k\,\parallelsum_{\Sigma_k}
    \end{tikzcd}\]
    
    \medskip
    \noindent where the middle vertical map is the decoupling map. This induces a map of the respective Serre spectral sequences $f:E^\bullet_{p,q}\to \Tilde{E}^\bullet_{p,q}$, and since $\tau$ is a homology isomorphism in degrees $i\leq \alpha$, the map between the $E^2$ pages
    
    \[\begin{tikzcd}
            {E^2_{p,q}=H_p\left(\Theta^m_0\parallelsum\Sigma_m\times(\Theta\parallelsum \GL_d^\dagger)_0^k\parallelsum{\Sigma_k} \; ;\; H_q(\mmm{M}^\Theta(W_{m+k},\rho_{W_{m+k}}))\right)} \arrow[d]
            \\
            {\tilde{E}^2_{p,q}=H_p\left(\Theta^m_0\parallelsum\Sigma_m\times (\Theta\parallelsum \GL_d^\dagger)_0^k\parallelsum{\Sigma_k} \; ;\; H_q(\mmm{M}^\Theta(W,\rho_{W}))\right)}
        \end{tikzcd}\]
    is an isomorphism for all $q\leq \alpha$. Then by the Spectral Sequence Argument (Lemma~\ref{specseqargument}), $D$ induces an isomorphism
    \[\begin{tikzcd}
            {H_i(\mmm{M}^{\Theta,k}_{m}(W,\rho_W)}) \arrow[r, "\cong"] &
            {H_i(\mmm{M}^\Theta(W,\rho_{W})\times \Theta^m_0\parallelsum\Sigma_m\times (\Theta\parallelsum \GL_d^\dagger)_0^k\parallelsum{\Sigma_k})}
    \end{tikzcd}\]
    for all $i\leq \alpha$.
\end{proof}

We now discuss how the decoupling result can be re-stated with a geometric interpretation. As discussed in Section \ref{sec:geometric-view}, the space  $\Emb(W,\mm{R}^\infty)$ is a model for $E\diff(W)$, and therefore it is also a model for $E\diff^k_m(W)$. With this model, the elements of $\mmm{M}^{\Theta,k}_m(W,\rho_W)$ are decorated submanifolds of $\mm{R}^\infty$ diffeomorphic to $W$ with $k$ marked points and $m$ disjoint embedded discs, with a choice of $\Theta$-structure concordant (ie. equivariantly homotopic) to $\rho_W$. With this model, the forgetful map 
    \begin{equation}
        F:\mmm{M}^{\Theta,k}_m(W,\rho_W)\to \mmm{M}^\Theta(W,\rho_W)
    \end{equation}
simply forgets the marked points and discs. 

To interpret the evaluation map $E$ with this model, we need also a geometric model for $E\Sigma_s$. Recall that the \emph{configuration space of $s$ points} in a manifold $M$ is defined as 
\[\conf_s(M):=\Emb(\{1,\dots,s\},M)/\Sigma_s\]
where the action of $\Sigma_s$ is given by permutation of the points in $\{1,\dots,s\}$. In other words, $\conf_s(M)$ is the space of unordered collections of $s$ distinct points in $M$. More generally, given a space $X$, the \emph{configuration space of $s$ points in M with labels in $X$} is defined as 
\[\conf_s(M;X):=(\Emb(\{1,\dots,s\},M)\times X^s)/\Sigma_s\]
where $\Sigma_s$ acts by permuting the factors of $X^s$, and acts on the product diagonally. In other words, $\conf_s(M)$ is the space of unordered collections of $s$ distinct points in $M$, where each point is labelled by a point in $X$.

Since the space $\Emb(\{1,\dots,m\},\mm{R^\infty})$ is weakly contractible, it is a model for the total space $E\Sigma_m$. Therefore a model for $\Theta^m\parallelsum \Sigma_m$ is precisely the space of unordered configurations of $m$ points in $\mm{R}^\infty$ with labels in $\Theta$. Analogously, a model for the $(\Theta\parallelsum \GL_d^\dagger)^k\parallelsum \Sigma_k$ is given by the space of unordered configurations of $k$ points in $\mm{R}^\infty$ with labels in $(\Theta\parallelsum \GL_d^\dagger)$.

Then the evaluation map $E$ takes a decorated submanifold $S$ in $\mm{R}^\infty$ together with a tangential structure $\rho$ to the configurations given by the centres of the $m$ marked points, and the $k$ marked discs. The labels of such configurations are determined by the tangential structure $\rho$: let $p$ be the centre point of a marked disc and let $V_p$ be the canonical frame of the tangent space of $W$ at $p$ induced by the parametrization of the disc. Then the label of the point corresponding to $p$ in the configuration space $\conf_m(\mm{R}^\infty;\Theta)$ is given by $\rho(V_p)\in\Theta$. Analogously, if $p$ is a marked point and $V_p$ is our chosen frame of its tangent space, then the label of the point corresponding to $p$ in $\conf_k(\mm{R}^\infty;\Theta\parallelsum\GL_d^\dagger)$ is simply the class of $\rho(V_p)$ in the Borel construction $\Theta\parallelsum\GL_d^\dagger$.

With these models, the decoupling map can be interpreted geometrically (see Figure \ref{fig:decoupling}), and the decoupling theorem can be re-stated as:

\begin{corollary}
    Let $W$ be a smooth connected compact manifold equipped with a $\Theta$-structure $\rho_W$. If the map $$\tau:H_i(\mmm{M}^\Theta(W_{m+k},\rho_{W_{m+k}}))\to H_i(\mmm{M}^\Theta(W,\rho_W))$$ induces a homology isomorphism in degrees $i\leq \alpha$, then for all such $i$ the decoupling map $D$ induces an isomorphism
    \[H_i(\mmm{M}^{\Theta,k}_{m}(W,\rho_W))\cong H_i(\mmm{M}^\Theta(W,\rho_W)\times \conf_m(\mm{R}^\infty;\Theta_0) \times \conf_k(\mm{R}^\infty;(\Theta\parallelsum \GL_d^\dagger)_0))\]
    where $(-)_0$ denotes a path component of the image of $\rho_W$, and the group $\GL_d^\dagger$ is $\GL_d^+$ if $W$ is orientable and \wellbehaved, and $\GL_d$ otherwise.
\end{corollary}

\section{Applications of the Decoupling Theorem}\label{section:corollaries}

\noindent The main hypothesis of the decoupling theorem is that the map 
\[\tau:H_i(\mmm{M}^\Theta(W_{m+k},r_{W_{m+k}}\rho_W))\to H_i(\mmm{M}^\Theta(W,\rho_W))\]
induces an isomorphism in a certain range $i\leq \alpha$. We know this condition holds in many circumstances for a variety of manifolds dimensions and tangential structures. In this section we recall some of these cases and discuss the applications of the decoupling theorem, focusing on the results for dimension $2$ and for even dimensions greater than $4$.

We remark that, for odd higher dimensions, many stability results on the homology of the moduli space have also been proven, but it is not yet known whether the map $\tau$ needed for the decoupling induces isomorphisms in a stable range.

\subsection{Applications for surfaces}

We now consider the $2$ dimensional case, where the stability holds in many circumstances. 

\subsubsection{Orientation}
For orientations, the classical stability result of Harer on the homology of mapping class groups of surfaces shows that the hypothesis of the decoupling theorem is satisfied for every oriented surface of genus $g$ and $b$ boundary components, $S_{g,b}$. The range in which the isomorphism holds has been improved throughout the years
\cite{Harer-orientable,ivanov1987complexes,ivanov1989stabilization,ivanov1993homology,boldsen2012improved,randal2016resolutions}. The most recent bound, by Randal-Williams in \cite{randal2016resolutions}, implies that the map $H_i(B\diff^+(S_{g,b+1}))\to H_i(B\diff^+(S_{g,b}))$ is an isomorphism for all $3i\leq 2g$. Then applying the decoupling theorem,  we recover the result of B\"odigheimer and Tillmann, now with an improved isomorphism range:

\begin{corollary}[\cite{MR1851247}]
For all $3i\leq 2g$
\[H_i(B\diff^{+,k}_m(S_{g,b}))\cong H_i(B\diff^+(S_{g,b})\times B\Sigma_m\times B(\Sigma_k\wr \SO(2))).\]
\end{corollary}

\begin{proof}
    This is a direct application of the decoupling theorem. In this case, $\Theta^{or}=\{\pm 1\}$ and therefore $\Theta^{or}_0=*$. Moreover, $\Theta^{or}\parallelsum\GL_2^+$ is the disjoint union of two copies of $B\GL_2^+\simeq B\SO(2)$, so $(\Theta^{or}\parallelsum\GL_2^+)_0\simeq B\SO(2)$, and therefore 
        \[(\Theta^{or}\parallelsum\GL_2^+)_0^k\parallelsum\Sigma^k\simeq B(\Sigma_k\wr\SO(2))\]
    as required.
\end{proof}

\medskip
\subsubsection{Non-orientable surfaces}
Let $\mathcal{N}_{g,b}$ be the decorated non-orientable surface $\#_g\mm{R}P^\infty\setminus\dcup{b}D^2$. Wahl showed in \cite{MR2367024} that the map $H_i(B\diff(\mmm{N}_{g,b+1}))\to H_i(B\diff(\mmm{N}_{g,b}))$ is an isomorphism for all $4i\leq g-3$. Applying the decoupling theorem, we recover the result of Hanbury in \cite{MR2439464}: 

\begin{corollary}[\cite{MR2439464}]
    For all $4i\leq g-5$
    \[H_i(B\diff^{k}_m(\mathcal{N}_{g,b}))\cong H_i(B\diff(\mathcal{N}_{g,b})\times B\Sigma_m\times B(\Sigma_k\wr \O(2))).\]
\end{corollary}

\begin{proof}
    The result follows from applying the decoupling theorem for $\Theta^*=*$. Then $\Theta^{*}\parallelsum\GL_2$ is homotopy equivalent to $B\GL_2\simeq B\O(2)$, and the result follows. 
\end{proof}

\medskip
\subsubsection{Framings} 
In \cite{MR3180616}, Randal-Williams showed that for the oriented surface $S_{g,b}$ with a framing $\rho$, the map $H_i(\mmm{M}^{\fr}(S_{g,b+1},\rho_{S_{g,b+1}}))\to H_i(\mmm{M}^{\fr}(S_{g,b},\rho))$ is an isomorphism for all $6i\leq 2g-8$.

\begin{corollary}
    Let $\rho$ be a framing on $S_{g,b}$, then for all $6i\leq 2g-8$
        \[H_i(\mmm{M}^{\mathrm{fr},k}_m(S_{g,b},\rho))\cong H_i(\mmm{M}^{\mathrm{fr}}(S_{g,b},\rho)\times \SO(2)^m\parallelsum\Sigma_m \times B\Sigma_k).\]
\end{corollary}

\begin{proof}
    The result is a direct application of the decoupling theorem, together with the fact that any path-component of $\Theta^{\fr}=\GL_2$ is homeomorphic to $\GL_2^+\simeq \SO(2)$ and that $\Theta^{\fr}\parallelsum\GL_2^+$ is equivalent to two points.
\end{proof}

\medskip
\subsubsection{Spin structures} 
In \cite{MR1054572,bauer2004infinite,MR3180616} it was shown that for the oriented surface $S_{g,b}$ with spin structure $\rho$, the map $H_i(\mmm{M}^{\Spin}(S_{g,b+1},\rho_{S_{g,b+1}}))\to H_i(\mmm{M}^{\Spin}(S_{g,b},\rho))$ is an isomorphism for all $5i\leq 2g-7$.

\begin{corollary}
    Let $\rho$ be a spin structure on $S_{g,b}$. For all $4i\leq g-2$, the group $H_i(\mmm{M}^{\mathrm{Spin},k}_{m}(S_{g,b},\rho))$ is isomorphic to    \[H_i\left(\mmm{M}^{\mathrm{Spin}}(S_{g,b},\rho)\times B(\Sigma_m\wr \mm{Z}/2)\times B(\Sigma_k\wr \Spin(2))\right).\]
\end{corollary}

\begin{proof}
    From Example \ref{ex:spin}, we know that $\Theta^\Spin$ is weakly equivalent to $\{\pm 1\}\times B\mm{Z}/2$, and therefore $(\Theta^\Spin)_0$ is weakly equivalent to $B\mm{Z}/2$. On the other hand, $\Theta^\Spin\parallelsum \GL_2^+$ is the disjoint union of two copies of $B\Spin(2)$ and therefore 
    \[(\Theta^\Spin\parallelsum \GL_2^+)_0^k\parallelsum\Sigma_k\simeq B\Spin(2)^k\parallelsum \Sigma_k\simeq B(\Sigma_k\wr \Spin(2)).\]
\end{proof}

\medskip
\subsubsection{Maps to a background space}
Consider the tangential structure given by maps to a simply-connected background space $X$. It was shown in  \cite{cohen06madsen,cohen10madsen,randal2016resolutions} that if $X$ is a simply-connected space, then the map $H_i(\mmm{M}^{\mathrm{X}}(S_{g,b+1},\rho_{S_{g,b+1}}))\to H_i(\mmm{M}^{\mathrm{X}}(S_{g,b},\rho))$ is an isomorphism for all $3i\leq 2g$, and all $\Theta_X$-structure $\rho$. Applying the decoupling theorem in this case, we obtain a generalisation of \cite[Theorem 8]{cohen10madsen}:

\begin{corollary}
    Let $\rho:S_{g,b}\to X$ be a continuous map and let $X_0$ be the path component containing the image of $\rho$. Then for all $3i\leq 2g$,
    \[H_i(\mmm{M}^{\mathrm{X},k}_{m}(S_{g,b},\rho))\cong H_i\left(\mmm{M}^{\mathrm{X}}(S_{g,b},\rho)\times (X_0)^m\parallelsum\Sigma_m \times (X_0)^k\parallelsum(\Sigma_k\wr \SO(2))\right).\]
\end{corollary}

\begin{proof}
    The result follows from applying the decoupling theorem together with the fact that $\Theta_X=X$ with the trivial action and $\Theta_X\parallelsum\GL_2^+\simeq X\times B\GL_2^+\simeq X\times B\SO(2)$.
\end{proof}

\medskip
\subsubsection{$\Spin^r$ structures}
The tangential structure called $\Spin^r$ is a generalisation of $\Spin$ which was thoroughly studied in \cite[Section 2]{MR3180616}. In this paper, Randal-Williams showed that for the oriented surface $S_{g,b}$ with a $\Spin^r$ structure $\rho$, the map $$H_i(\mmm{M}^{\Spin^r}(S_{g,b+1},\rho_{S_{g,b+1}}))\to H_i(\mmm{M}^{\Spin^r}(S_{g,b},\rho))$$ is an isomorphism for all $6i\leq 2g-8$.

\begin{corollary}
    Let $\rho$ be a $\Spin^r$ structure on $S_{g,b}$. For all $6i\leq 2g-8$, the group $H_i(\mmm{M}^{\mathrm{Spin^r},k}_{m}(S_{g,b},\rho))$ is isomorphic to
    \[H_i\left(\mmm{M}^{\mathrm{Spin^r}}(S_{g,b},\rho)\times B(\Sigma_m\wr\mm{Z}/r)\times B(\Sigma_k\wr\Spin^r(2))\right).\]
\end{corollary}

\begin{proof}
    Using the fibre sequence 
    \[\begin{tikzcd}[row sep=tiny]
            \{\pm 1\}\times B\mm{Z}/r\ar[r]               & B\Spin^r(2)       \ar[r]        & B\GL_2\\
        \end{tikzcd}\]
    we can apply the procedure described in Example~\ref{ex:spin} to deduce that a path component of $\Theta^{\Spin^r}$ is weakly equivalent to the Lens space $B\mm{Z}/r$ and $\Theta^{\Spin^r}\parallelsum\GL_2^+$ is the disjoint union of two copies of $B\Spin^r(2)$.
\end{proof}

\medskip
\subsubsection{$\Pin^{\pm}$ structures}
The tangential structures called $\Pin^{+}$ and $\Pin^-$ are generalisations of $\Spin$ for non-orientable manifolds, and they were thoroughly studied in \cite[Section 4]{MR3180616}. In this paper, Randal-Williams showed that for the non-orientable surface $\mmm{N}_{g,b}=\#_g\mm{R}P^\infty\setminus\dcup{b}D^2$ with a $\Pin^{+}$-structure, the map $$H_i(\mmm{M}^{\Pin^+}(\mmm{N}_{g,b+1},\rho_{\mmm{N}_{g,b+1}}))\to H_i(\mmm{M}^{\Pin^+}(\mmm{N}_{g,b},\rho))$$ is an isomorphism for all $4i\leq g-6$, and $\Pin^{+}$-structure $\rho$.

It was also shown in \cite{MR3180616} that the map $$H_i(\mmm{M}^{\Pin^-}(\mmm{N}_{g,b+1},\rho_{\mmm{N}_{g,b+1}}))\to H_i(\mmm{M}^{\Pin^-}(\mmm{N}_{g,b},\rho))$$ is an isomorphism for all $5i\leq g-8$, and $\Pin^{-}$-structure $\rho$.

\begin{corollary}\label{cor:pin+}
    Let $\rho$ be a $\Pin^+$ structure on $\mathcal{N}_{g,b}$, then for all $4i\leq g-6$
    \[H_i(\mmm{M}^{\mathrm{Pin}^+,k}_m(\mathcal{N}_{g,b},\rho))\cong H_i(\mmm{M}^{\mathrm{Pin}^+}(\mathcal{N}_{g,b},\rho)\times B(\Sigma_m\wr \mm{Z}/2)\times B(\Sigma_k\wr\Pin^+(2))).\]
\end{corollary}

\begin{corollary}\label{cor:pin-}
     Let $\rho$ be a $\Pin^-$ structure on $\mathcal{N}_{g,b}$, then for all $5i\leq g-8$
     \[H_i(\mmm{M}^{\mathrm{Pin}^-,k}_m(\mathcal{N}_{g,b},\rho))\cong H_i(\mmm{M}^{\mathrm{Pin}^-}(\mathcal{N}_{g,b},\rho)\times B(\Sigma_m\wr \mm{Z}/2)\times B(\Sigma_k\wr\Pin^-(2))).\]
\end{corollary}

\begin{proof}[Proof of Corollaries \ref{cor:pin+} and \ref{cor:pin-}]
    Using the fibre sequence 
    \[\begin{tikzcd}[row sep=tiny]
        B\mm{Z}/2 \ar[r]               & B\Pin^{\pm}(d)       \ar[r]        & B\GL_d
    \end{tikzcd}\]
    we can apply the procedure described in Example~\ref{ex:spin} to deduce that a path component of $\Theta^{\Pin^{\pm}}$ is weakly equivalent to $B\mm{Z}/2$, and $\Theta^{\Pin^{\pm}}\parallelsum\GL_2$ is the space $B\Pin^{\pm}(2)$.
\end{proof}

\medskip
\subsection{Applications for high dimensional manifolds}\label{subsec:app-high-dim}
One of the most interesting applications of the decoupling result appears when looking at higher dimensional manifolds. In \cite{MR3665002}, it was shown that the hypothesis of the decoupling theorem holds for many manifolds $W$ of even dimension greater or equal to $6$, and many tangential structures. The range of the homology isomorphism is given in terms of the stable genus of $W$, which we now recall, following the notation of \cite{galatius2018moduli}.

Analogously to the surface case, the genus will be measured by disjoint embeddings of the space $(S^n\times S^n)\setminus\{*\}$, but now taking into account the tangential structure as well. Namely, Galatius and Randal-Williams define what it means for a $\Theta$-structure on $(S^n\times S^n)\setminus \{*\}$ to be \emph{admissible} (see \cite[Section 3.2]{galatius2018moduli}) and define the \emph{genus} of a manifold $W$ with $\Theta$-structure $\rho_W$ to be 
\[g(W,\rho_W)=\max\left\{g\in\mm{N}\left|
\begin{tabular}{c}
     \text{\small there are $g$ disjoint embeddings $j:(S^n\times S^n)\setminus \{*\}\hookrightarrow W$}\\ \text{\small such that $j^*\rho_W$ is admissible}
\end{tabular}
\right.\right\}.\]

The \emph{stable genus} of $(W,\rho_W)$ is defined to be
\[\overline{g}(W,\rho_W)=\max\left\{g\left(W\# W_{k,1},\rho^{(k)}_W\right)-k|k\in \mm{N}\right\}\]
where $W\# W_{k,1}$ is obtained from $W$ by removing $k$ discs and attaching $k$ copies of $(S^n\times S^n)\setminus\int(D^{2n})$ along the new boundary. The $\Theta$-structure $\rho^{(k)}_W$ is obtained by extending the restriction of $\rho_W$ by any admissible structure on $(S^n\times S^n)\setminus\int(D^{2n})$.

\begin{lemma}\label{lemma:genus}
    Let $W$ be a smooth compact manifold of dimension $2n\geq 2$, $L\subset \int(W)$ a closed submanifold of dimension $\leq n-1$, and $N$ a tubular neighbourhood of $L$. Then the genus of $W\setminus N$ is equal to the genus of $W$.
\end{lemma}

\begin{proof}
    Sard's theorem implies that for any submanifold $L' \subset (S^n\times S^n)\setminus \{*\}$ with $dim(L') \leq n-1$ there is an embedding $(S^n\times S^n)\setminus \{*\} \hookrightarrow (S^n\times S^n)\setminus \{*\}$ that avoids $L'$ and is isotopic to the identity. In particular, this implies that for any $\phi:\dcup{g}(S^n\times S^n)\setminus \{*\}\hookrightarrow W$, there is an embedding 
    \[\phi':\begin{tikzcd}\dcup{g}(S^n\times S^n)\setminus \{*\} \ar[r, hook] & \dcup{g}(S^n\times S^n)\setminus \{*\} \ar[r, hook, "\phi"]& W \end{tikzcd}\]
    that avoids $L$ and is isotopic to $\phi$. Since $W\setminus L$ is diffeomorphic to $\int(W\setminus N)$ via a diffeomorphism fixing everything but a collar of $L$, the result follows.
\end{proof}

Let $W$ be a manifold with non-empty boundary $P$, and $\rho_W$ a $\Theta$-structure on $W$. Given $M$ a cobordism from $P$ to $Q$ together with a $\Theta$-structure $\rho_M$ on $M$ which restricts to $\rho_W$ over $P$, there is an induced map
    \begin{equation}\label{eq:stability-map}
        {-\cup _P (M,\rho_M)}:\begin{tikzcd}
        \mmm{M}^\Theta(W,\rho_W) \ar[r] & \mmm{M}^\Theta(W\cup_P M,\rho_W\cup \rho_M)
        \end{tikzcd}
    \end{equation}
which is induced by the $\diff(W)$-equivariant map $\Bun^\Theta(W,\rho_W)\to \Bun^\Theta(W\cup_P M,\rho_W\cup \rho_M)$ given by extending a map by $\rho_M$, and the homomorphism $\diff(W)\to \diff(W\cup_P M)$ given by extending a map by the identity on $M$.

\begin{theorem}[\cite{MR3665002}, Corollary 1.7]\label{thm:hom-st-high-dim}
    Assume $d=2n\geq 6$, and $\Theta$ is such that $\Theta\parallelsum\GL_d$ is simply-connected. Let $\rho_W$ be an $n$-connected $\Theta$-structure on $W$ and let $g=\overline{g}(W,\rho_W)$. Given a cobordism $(M,\rho_M)$ as above such that $(M,P)$ is $(n-1)$-connected, the map 
        \begin{equation*}
         {(-\cup _P (M,\rho_M))_*}:
        \begin{tikzcd}
        H_i(\mmm{M}^\Theta(W,\rho_W)) \ar[r] & H_i(\mmm{M}^\Theta(W\cup_P M,\rho_{W\cup_P M}))
        \end{tikzcd}
    \end{equation*}
    is an isomorphism for all $3i\leq g-4$.
\end{theorem}

We recall that a map is called $n$-connected if the map induced on homotopy groups $\pi_i$ is an isomorphism for $i<n$ and a surjection for $i=n$.

\begin{remark}
    In \cite[Corollary 1.7]{MR3665002}, the result above is given in much more generality, allowing arbitrary coefficient systems and providing a better stability range depending on the coefficient system and the tangential structure. We restrict ourselves to the case above, for simplicity, but remark that such generalisations can also be immediately carried out in the decoupling theorem.
\end{remark}

\begin{corollary}\label{higherdimensions}
    Assume $d=2n\geq 6$, and $\Theta$ is such that $\Theta\parallelsum\GL_d$ is simply-connected. Let $\rho_W$ be an $n$-connected $\Theta$-structure on $W$ and let $g=\overline{g}(W,\rho_W)$. Then for all $i\leq \frac{g-4}{3}$ we have an isomorphism
    \[H_i(\mmm{M}^{\Theta,k}_{m}(W,\rho_W))\cong H_i(\mmm{M}^\Theta(W,\rho_W)\times \Theta^m_0\parallelsum\Sigma_m \times (\Theta\parallelsum \GL_{2n}^\dagger)_0^k\parallelsum{\Sigma_k})\]
    where $\GL_{2n}^\dagger$ equals to $\GL_{2n}^+$ if $W$ is \wellbehaved, and is $\GL_{2n}$ otherwise.
\end{corollary}

\begin{proof}
    First notice that since the map $\rho_W$ is $n$-connected and the pair $(W,W_{m+k})$ is $(2n-1)$-connected, then the restriction $\rho_{W_{m+k}}$ is still an $n$-connected $\Theta$-structure.
    
    The map $\tau:H_i(\mmm{M}^\Theta(W_{m+k},\rho_{W_{m+k}}))\to H_i(\mmm{M}^\Theta(W,\rho_W))$ in the hypothesis of the decoupling theorem, is induced by attaching $\dcup{m+k}D^{2n}$ along the $m+k$ boundary sphere components of $W_{m+k}$. Since $$(M,P)=(\dcup{m+k}D^{2n},\partial \dcup{m+k}D^{2n})$$ is $(n-1)$-connected, the hypotheses of Theorem \ref{thm:hom-st-high-dim} are satisfied, which implies that $\tau$ induces a homology isomorphism in degrees $3i\leq g-4$. Applying Theorem \ref{decoupling}, the result follows.
\end{proof}

\medskip

\begin{example}\label{ex:Wg1-framed}
    Let $W_{g,1}=(S^n\times S^n)\# D^{2n}$. Since $TW_{g,1}$ is trivialisable, we know $W_{g,1}$ admits a framing $\rho_{W_{g,1}}:\Fr(TW_{g,1})\to \GL_{2n}$ fitting into the following pullback diagram:
        \[\begin{tikzcd}
            \Fr(TW_{g,1}) \ar[r, "\rho_{W_{g,1}}"] \ar[d]  \arrow[dr,phantom, "\lrcorner", very near start] & \GL_{2n} \ar[d]\\
            W_{g,1} \ar[r] & E\GL_{2n}
        \end{tikzcd}\]
    The bottom arrow is necessarily $n$-connected because $W_{g,1}$ is $(n-1)$-connected and $E\GL_{2n}$ is weakly contractible. Therefore, $\rho$ is $n$-connected as well.
    
    Let $\overline{g}$ denote the stable genus $\overline{g}(W_{g,1},\rho_{W_{g,1}})$. By Corollary~\ref{higherdimensions}, for all $i\leq \frac{\overline{g}-4}{3}$, the group $H_i(\mmm{M}^{\mathrm{fr},k}_m(W_{g,1},\rho_{W_{g,1}}))$ is isomorphic to
    \[H_i\left(\mmm{M}^\mathrm{fr}(W_{g,1},\rho_{W_{g,1}})\times \SO(2n)^m\parallelsum\Sigma_m\times B\Sigma_k\right).\]
\end{example}
\medskip

\begin{example}\label{ex:Vd}
    Let $V_d\subset \mm{C}P^{4}$ be a smooth hypersurface determined by a homogeneous complex polynomial of degree $d$. This is an orientable chiral $6$-dimensional manifold whose diffeomorphism type depends only on the degree $d$. In section 5.3 of \cite{galatius2018moduli}, Galatius and Randal-Williams show that, if $d$ is even, there exists a $3$-connected $\Spin^c$-structure $\rho_{V_d}$ on $V_d$. They also compute an expression for the stable genus $\overline{g}(V_d,\rho_{V_d})$ in terms of $d$. 
    
    Applying the procedure of Example~\ref{ex:spin} to the fibre sequence
    \[\begin{tikzcd}[row sep=tiny]
        \{\pm 1\}\times B\U(1)\ar[r]               & B\Spin^c(d)       \ar[r]        & B\GL_d\\
    \end{tikzcd}\]
    we get that $\Theta^{\Spin^c}\simeq\{\pm1\}\times B\U(1)$, and $\Theta^{\Spin^c}\parallelsum \GL_6^{+}\simeq\{\pm1\}\times B\Spin^c(6)$.
    Therefore, by Corollary~\ref{higherdimensions}, for all $i\leq \frac{d^4-5d^3+10d^2-10d+4}{4}$, the group $H_i(\mmm{M}^{\mathrm{Spin^c},k}_m(V_d,\rho_{V_d}))$ is isomorphic to
    \[H_i\left(\mmm{M}^\mathrm{Spin^c}(V_d,\rho_{V_d})\times B(\Sigma_m\wr\U(1))\times B(\Sigma_k\wr\Spin^c(6))\right)\]
\end{example}
\medskip

The conditions on the tangential structure in Theorem~\ref{thm:hom-st-high-dim} are quite restrictive, for instance the trivial tangential structure $\Theta^*$ does not satisfy the hypothesis because $B\GL_d$ is not simply connected for any $d$. Moreover, the condition that we start with an $n$-connected $\Theta$-structure $\rho_W$ excludes many of the cases we are interested in. For instance, it implies that the manifold $W_{g,1}$ with an orientation does not satisfy the hypothesis of Theorem~\ref{thm:hom-st-high-dim}. However, in \cite[Section 9]{MR3665002}, Galatius and Randal-Williams provided a generalisation of this result for general tangential structures. In Section~\ref{section:general-tangential-structures}, we use their techniques to prove a generalisation of Theorem~\ref{main-decoupling} for high dimensional manifolds with any tangential structure. 
\section{Decoupling Submanifolds}\label{section:decoupling-submanifolds}

\noindent In Section~\ref{section:proof-of-decoupling}, we proved a decoupling result for the decorated moduli space of a manifold with marked points and discs, following the works of \cite{MR1851247,MR2439464,MR2861233}. Recently, in \cite{palmer2012configuration,palmer2018homological,palmer2018homological2} Palmer has studied manifolds equipped with more general decorations, allowed to be any embedded closed manifold $P$. In this section, we show that there is a decoupling result for these generalised decorations. As a specific example, we focus on the case where the decorations are unlinked circles, which have also been closely studied in dimension $3$ by Kupers in \cite{kupers2013homological}.

\subsection{The $L$-decorated moduli space}\label{sec:decorated-sub-moduli-space} In this section, we generalise the definition of a decorated manifold to allow more general submanifolds as decorations. Throughout, let $W$ be a smooth connected compact $d$-dimensional manifold.

\begin{definition}\label{Ldecorations}
    A $d$-dimensional \emph{$L$-decorated manifold} is a pair $(W,L)$ of a manifold $W$ together with a closed submanifold $L\subset W$. 

    Given a $L$-decorated manifold $(W,L)$, we define the \emph{decorated diffeomorphism group} $\diff_L(W)$ to be the subgroup of $\diff(W)$ of the diffeomorphisms $\psi$ such that  
        \begin{align*}
            \psi(L)=L.
        \end{align*}
\end{definition}

In other words, we are looking at the diffeomorphisms preserving the marked submanifold, but not necessarily pointwise.

\begin{definition}
    Given a closed manifold $W$ and a $\Theta$-structure $\rho_W$ on $W$, we define the \emph{$L$-decorated moduli space of $W$} to be
    \begin{align*}
        \mmm{M}^{\Theta}_L(W,\rho_W):=\Bun^{\Theta}(W,\rho_W)\parallelsum \diff_L(W).
    \end{align*} 
\end{definition}

The inclusion of groups $\diff_L(W)\to \diff(W)$ induces a map
    \begin{equation}\label{eq:forgetful-map-submanifold}
        F_L:\mmm{M}^{\Theta}_L(W,\rho_W)\to \mmm{M}^\Theta(W\rho_W)
    \end{equation}
which we call the \emph{forgetful map}.

\subsection{The evaluation map $E_L$}
Let $(W,L)$ be an $L$-decorated manifold and let $\nu_{L}\coloneqq(TW_{|L})/TL$ be the normal bundle of $L$ in $W$. Let $N$ be the tubular neighbourhood of the decoration identified as the image of an embedding $\Phi:\nu_{L}\to W$, and denote by $W_{N}$ the manifold $W\setminus N$. The theorem for decoupling submanifolds relies on understanding the difference between the $L$-decorated moduli space of $W$ and the moduli space of $W_N$. 

We start by constructing an equivariant fibre sequence relating the decorated diffeomorphism groups $\diff_L(W)$ and $\diff(W_N)$. Recall that $\diff(W_N)$ consists only of those diffeomorphisms fixing a collar neighbourhood of the boundary of $W_N$, including the newly formed boundary obtained by removing $N$. Extending a diffeomorphism by the identity on $N$, gives us a homomorphism 
    $$ \diff(W_N) \to \diff_L(W).$$

On the other hand, since any diffeomorphism $\phi\in \diff_L(W)$ fixes $L$, the differential of $\phi$ induces an isomorphism of the tangent bundle $TW_{|L}$ fixing $TL$ (not necessarily pointwise). This gives a map:
    \begin{align}\label{evaluation-normal-bundle}
        e_L:\diff_L(W) & \longrightarrow \iso(TW_{|L},TL)\\
        \phi  &\longmapsto   D\phi|_{L}
    \end{align}
where $D\phi|_{L}$ denotes the isomorphism of $TW_{|L}$ induced by the differential of $\phi$, and  $\iso(TW_{|L},TL)$ denotes the group of bundle isomorphisms of $TW_{|L}$ fitting into the following diagram:
    \[\begin{tikzcd}
        TL \ar[r, "Df"] \ar[d] & TL \ar[d]\\
        TW_{|L} \ar[r, "\overline{f}"] \ar[d] & TW_{|L} \ar[d]\\
        L \ar[r, "f"] & L.
    \end{tikzcd}\]

\begin{definition}
    For any subgroup $G\subset \ima e_L$, we define $\diff_G(W)$ to be the subgroup $e_L^{-1}(G)$. Given a closed manifold $W$ and a $\Theta$-structure $\rho_W$ on $W$, we define 
    \begin{align*}
        \mmm{M}^{\Theta}_G(W,\rho_W):=\Bun^{\Theta}(W,\rho_W)\parallelsum \diff_G(W).
    \end{align*} 
\end{definition}

Note that taking $G=\ima e_L$, one recovers precisely the definition of $\mmm{M}^\Theta_L(W,\rho_W)$. The kernel of $e_L$ consists precisely of those elements of $\diff_L(W)$ which fix the submanifold $L$ pointwise and whose differential $D_p\phi$ is the identity on every point of the submanifold $L$. We denote the kernel of $e_L$ by $\diff(W,TW|_L)$. 

\begin{lemma}\label{diffgroupsL}
    Let $(W,L)$ be a compact connected $L$-decorated manifold, then
    \begin{enumerate}[(a)]
        \item the homomorphism 
            \[e_L:\begin{tikzcd}\diff_L(W) \arrow[r] & \iso(TW_{|L},TL)\end{tikzcd}\]
        is a principal bundle;
        \item the map
            \[i:\begin{tikzcd}\diff(W_N) \arrow[r] & \diff(W,TW|_L)\end{tikzcd}\]
        is a homotopy equivalence.
    \end{enumerate}
\end{lemma}

\begin{proof}
    Throughout this proof, we will use a generalisation of Palais' theorem in \cite{MR0123338} proved by Lima in \cite{MR0161343}, which gives us a principal bundle
        $$\begin{tikzcd}
            \diff(W_N) \arrow[r] & \diff(W) \arrow[r] & \Emb(N,W).
        \end{tikzcd}$$
    Let $\Emb_L(N,W)$ be the subspace of embeddings $f:N\hookrightarrow W$ such that the core of $N$ is taken to our marked submanifold $L$ in $W$. Then taking the pullback along the inclusion $\Emb_L(N,W)\hookrightarrow \Emb(N,W)$ gives as the principal bundle:
        \begin{equation}\label{eq:fibre-sequence-emb-L}
            \begin{tikzcd}
            \diff(W_N) \arrow[r] & \diff_L(W) \arrow[r, "r"] & \Emb_L(N,W)
        \end{tikzcd}
        \end{equation}
    
     Write $N$ as the image of an embedding $\exp\circ\Phi:\nu_L\hookrightarrow W$, where $\Phi:\nu_L\to TW_{|L}$. Consider the forgetful map
        \[d:\Emb_{L}(N,W)\rightarrow \iso(TW_{|L},TL)\]
    taking an embedding to the map induced on the normal bundle of the zero section $L\subset N$. Then $e_L=d\circ r$. We will show $d$ is a fibre bundle, which implies $e_L$ is a principal bundle. It is enough to exhibit a local section of $d$ at a neighbourhood of the identity (for details see \cite[Part I, Section 7.4]{steenrod1999topology}).
    
    Given $\overline{f}:TW_{|L}\to TW_{|L}$ in $\iso(TW_{|L},TL)$ we can define 
        \[s_f:\begin{tikzcd}
            \nu_L \ar[r, "\Phi"] & TW_{|L} \ar[r, "\overline{f}"] & TW_{|L} \ar[r, "\exp"] & W.
        \end{tikzcd}\]
    Since the assignment $f\mapsto s_f$ is continuous and $\Emb(\nu_L,W)$ is an open subset of $\mmm{C}^\infty(\nu_L,W)$, then the space of maps $\overline{f}\in \iso(TW_L,TL)$ such that $s_f$ is an embedding, is an open neighbourhood $U$ of the identity. Therefore, the map 
        \begin{align*}
            U &\longrightarrow \Emb(\nu_L,W)\\
            \overline{f} &\longmapsto s_f
        \end{align*}
    is a local section for $d$ at the identity.
    
     \emph{Part (b):} We will show the map  $i$ is a homotopy equivalence, from the fact that it fits into the following commutative diagram of fibre sequences:
        $$\begin{tikzcd}
            \diff(W_N) \arrow[r] \arrow[d, "i"] & \diff_L(W) \arrow[r] \arrow[d, "="] & \Emb_L(N,W) \arrow[d, "d"]\\
            \diff(W,TW|_L) \arrow[r] & \diff_L(W) \arrow[r, "e_L"] & \iso(TW_{|L},TL)
        \end{tikzcd}$$
    The fiber of the forgetful map $d$ over the identity is simply the space of tubular neighbourhoods of $L$ in $W$, which is contractible. This implies $d$ is a homotopy equivalence and therefore so is $i$.
\end{proof}

We now use the map $e_L$ and Lemma~\ref{diffgroupsL} to construct the evaluation map:

\begin{proposition}\label{thefibrationL}
    Let $W$ be a compact connected manifold and $\rho_W$ a fixed $\Theta$-structure on $W$, then:
    \begin{enumerate}[(a)]
        \item\label{itm:EL-is-fibration} the homomorphism $e_L$, induces an {evaluation map}
            \[E_L:\begin{tikzcd} \mmm{M}^{\Theta}_G(W,\rho_W) \ar[r] & \Map_{\GL_d}(\Fr(TW|_L),\Theta)\parallelsum G \end{tikzcd}\]
        which is a Serre fibration onto the path component which it hits.
        \item\label{itm:homotopy-fibre-EL} Let $W_N$ be equipped with the $\Theta$-structure $\rho_{W_N}$ given by the restriction of $\rho_W$. Then $$\mmm{M}^\Theta(W_N,\rho_{W_N})$$ is the homotopy fibre of $E_L$ over its image.
    \end{enumerate}
\end{proposition}

\begin{proof}
    \ref{itm:E-is-fibration} By Lemma \ref{bun-fibration}, the restriction map
        \[r_L:\Bun^\Theta(W,\rho_W)\to \Map_{\GL_d}(\Fr(TW|_L),\Theta)\]
    is a Serre fibration. Then the result follows by applying Corollary~\ref{fibration-descends-to-borel} to combine the fibration $r_L$ with the homomorphism
        \begin{equation}
            e_L:\diff_G(W)\rightarrow  G.
        \end{equation}
    We apply Corollary~\ref{fibration-descends-to-borel}\ref{itm:fibration-descends-b} by taking $G_2=\diff_G(W)$ and $S_2=\Bun^\Theta(W,\rho_W)$, with the usual action by precomposition with the differential. On the other hand, we take $G_3=G$, and $S_3= \Map_{\GL_d}(\Fr(TW|_L),\Theta)$ with the action induced by 
        \[\begin{tikzcd}[column sep=tiny] \iso(TW_{|L},TL) & \Fr(TW|_L). \ar[loop left, out=190, in= 170, distance=10] \end{tikzcd}\]
    Then the fibration $\Bun^\Theta(W,\rho_W)\to \Map_{\GL_d}(\Fr(TW|_L),\Theta)$ is $e_L$-equivariant and therefore, by Corollary~\ref{fibration-descends-to-borel}\ref{itm:fibration-descends-b}, we have a fibration
        \[\begin{tikzcd} \mmm{M}^{\Theta}_G(W,\rho_W) \ar[r, "E_L"] & \Map_{\GL_d}(\Fr(TW|_L),\Theta)\parallelsum G \end{tikzcd}\]
    onto the path components which it hits.
    
    \medskip
    \ref{itm:homotopy-fibre-EL} Recall that $W_N$ is defined as the submanifold $W\setminus N$, where $N$ is a tubular neighbourhood of the submanifold $L$. Then the restriction $\rho_{W_N}$ is a $\Theta$-structure on $W_N$. In the remainder of the proof, we will show that $\mmm{M}^\Theta(W_N,\rho_{W_N})$ is the homotopy fibre of $E_L$.
    
    A description of the fibre of $E_L$ can be obtained using Corollary~\ref{fibration-descends-to-borel}\ref{itm:fibre-over-descent-b} with the short exact sequence of groups being 
        \begin{equation}\label{eq:prop-ses-L}
            \ker e_L \to \diff_G(W)\xrightarrow{e_L}  G
        \end{equation}
    and the fibre sequence of $S_1\to S_2\to S_3$ being the one associated to the fibration $r_L$. The fibre of $r_L$ over $r_L(\rho_W)$ is the subspace of all elements of $\Bun^\Theta(W,\rho_W)$ which restrict to $r_L(\rho_W)$ over $L$, which we here denote $\Bun^\Theta_L(W,\rho_W)$. This space carries an action of $\ker e_L$ by precomposition with the differential, and it is simple to check that this fibre sequence is equivariant with respect to \eqref{eq:prop-ses-L}. Then by Corollary~\ref{fibration-descends-to-borel}\ref{itm:fibre-over-descent-b}, the fibre of the evaluation map $E_L$ is given by 
        \[\Bun^\Theta_L(W,\rho_W)\parallelsum\ker e_L.\]
    
    Applying Lemma \ref{bun-fibration} to both submanifolds $L$ and $N$, we obtain two fibrations fitting into the following commutative diagram
        \[\begin{tikzcd}
            \Bun^\Theta(W,\rho_W) \ar[r, "r_N"] \ar[d, "="] & \Map_{\GL_d}(\Fr(TW|_N),\Theta) \phantom{\cong \Theta^N} \ar[d, "i^*"]\\
            \Bun^\Theta(W,\rho_W) \ar[r, "r_L"] & \Map_{\GL_d}(\Fr(TW|_L),\Theta) 
        \end{tikzcd}\]
    where the right-hand vertical map is induced by the inclusion $i:L\hookrightarrow N$. Since $i$ is a strong deformation retract, the map $i^*$ is a homotopy equivalence.
    In particular, this implies that the map from the fibre of $r_N$ to $\Bun^\Theta_L(W,\rho_W)$ is a homotopy equivalence. 
    
    The fibre of $r_{_N}$ over $r_{_N}(\rho_W)$ is by definition the space of all $\Theta$ structures on $W$ which agree with $\rho_W$ on $N$. We claim that this space is homeomorphic to $\Bun^\Theta(W_N,\rho_{W_N})$, since the restriction map $r_{_{W_N}}$ takes the fibre of $r_{_N}$ bijectively to $\Bun^\Theta(W_N,\rho_{W_N})$ and it has an inverse given by extending an element by $r_{_N}\rho_{W}$.
    
    So we have a commutative diagram of principal fibre bundles
        \[\begin{tikzcd}
            \diff(W_N) \ar[r, "\simeq"] \ar[d] & \ker e_L \ar[d]\\
            \Bun^\Theta(W_N,\rho_{W_N})\times E\diff(W) \ar[r, "\simeq"] \ar[d] & \Bun^\Theta_L(W,\rho_W)\times E\diff(W) \ar[d]\\
            \mmm{M}^\Theta(W_N,\rho_{W_N}) \ar[r] & \Bun^\Theta_L(W,\rho_W)\parallelsum \ker e_L
        \end{tikzcd}\]
    where the top horizontal map is a homotopy equivalence by Lemma \ref{diffgroupsL} and the middle map is a homotopy equivalence by the discussion above. Therefore the map 
        \[\mmm{M}^\Theta(W_N,\rho_{W_N}) \to \Bun^\Theta_L(W,\rho_W)\parallelsum \ker e_L\]
    is also a homotopy equivalence, as required.
\end{proof}

Proposition \ref{thefibration} can be recovered as a special case of Proposition  \ref{thefibrationL}: let $L$ be comprised of $m+k$ points (which are the $k$ marked points and the centres of the $m$ marked discs), then $TL$ is a zero-dimensional bundle and $TW_{|L}$ is a trivial bundle of dimension $d$. Defining $G\subset \iso(TW_{|L},TL)\cong \iso(\dcup{m+k}\mm{R}^d)$ to be the subgroup $(\Sigma_k\wr\GL_d)\times \Sigma_m$, we recover precisely the case analysed in Proposition \ref{thefibration}. Note that an element of $\diff_G(W)$ can permute the $k$ marked points with no restrictions on the map induced on their tangent bundle, on the other hand, the $m$ points are allowed to be permuted, but the map induced on their tangent spaces has to be the identity. 

\subsection{Decoupling $L$-decorations}
In this section we prove the decoupling result by comparing a homotopy fibre sequence constructed in Proposition \ref{thefibrationL} to the product fibre sequence via the aforementioned spectral sequence argument.

\begin{definition}
    The \emph{decoupling map}
    \[D_L:\begin{tikzcd}
        \mmm{M}^{\Theta}_G(W,\rho_W) \ar[rr, "F_L\times E_L"] && \mmm{M}^\Theta(W,\rho_W)\times \Map_{\GL_d}(\Fr(TW|_L,\Theta))\parallelsum G
    \end{tikzcd}\]
    is the product of the forgetful map \ref{eq:forgetful-map-submanifold} and the evaluation map $E_L$ defined in Proposition~\ref{thefibrationL}.
\end{definition}

We now state the decoupling theorem:

\begin{theorem}\label{decouplingL}
    Let $(W,L)$ be an $L$-decorated manifold, with $W$ a connected compact manifold equipped with a $\Theta$-structure $\rho_W$, and $G\subset \ima e_{L}$. If $\tau:H_i(\mmm{M}^\Theta(W_N,\rho_{W_N}))\to H_i(\mmm{M}^\Theta(W,\rho_W))$ is an isomorphism in degrees $i\leq \alpha$, then for all such $i$ the decoupling map $D_L$ induces an isomorphism
    \[H_i(\mmm{M}^{\Theta}_G(W,\rho_W))\cong H_i(\mmm{M}^\Theta(W,\rho_W)\times (\Map_{\GL_d}(\Fr(TW|_L),\Theta)\parallelsum G)_0)\]
    where $(-)_0$ denotes a path component of $E_L(\rho_W)$.
\end{theorem}

\medskip
\begin{proof}
    By Proposition \ref{thefibration}, $E_L$ is a fibration onto the path-components which it hits, therefore the restriction of $E_L$ to the subspace $\mmm{M}^{\Theta}_G(W,\rho_W)$ is a fibration onto the path-component of $\Map_{\GL_d}(\Fr(TW|_L),\Theta)\parallelsum G$ which it hits. We denote it 
        \[(\Map_{\GL_d}(\Fr(TW|_L),\Theta)\parallelsum G)_0.\]
    Therefore, we have a homotopy fibre sequence
    \[\begin{tikzcd}
        {\mmm{M}^\Theta(W_N,\rho_{W_N})} \arrow[r] & {\mmm{M}^{\Theta}_G(W,\rho_W)} \arrow[r, "E_L"] & (\Map_{\GL_d}(\Fr(TW|_L),\Theta)\parallelsum G)_0
    \end{tikzcd}\]
    The proof of the theorem follows from the comparison between this homotopy fibre sequence and the one associated to the trivial fibration
    \[\begin{tikzcd}
        {\mmm{M}^\Theta(W,\rho_{W})\times (\Map_{\GL_d}(\Fr(TW|_L),\Theta)\parallelsum G)_0} \arrow[r] & (\Map_{\GL_d}(\Fr(TW|_L),\Theta)\parallelsum G)_0
    \end{tikzcd}\]
    By definition of the maps in Proposition~\ref{thefibrationL}, the following is a commutative diagram of homotopy fibre sequences 
    
    \medskip
    
    \[\begin{tikzcd}[column sep=tiny]
        {\mmm{M}^\Theta(W_N,\rho_{W_N})} \arrow[r] \arrow[d, "\tau"] & {\mmm{M}^{\Theta}_G(W,\rho_W)} \arrow[r, "E_L"] \arrow[d, "D_L"] & (\Map_{\GL_d}(\Fr(TW|_L),\Theta)\parallelsum G)_0    \arrow[d, equal] \\
        {\mmm{M}^\Theta(W,\rho_{W})} \arrow[r] & {\mmm{M}^\Theta(W,\rho_{W})\times (\Map_{\GL_d}(\Fr(TW|_L),\Theta)\parallelsum G)_0} \arrow[r] & (\Map_{\GL_d}(\Fr(TW|_L),\Theta)\parallelsum G)_0
    \end{tikzcd}\]
    
    \medskip
    \noindent where the middle vertical map is the decoupling map. This induces a map of the respective Serre spectral sequences $f:E^\bullet_{p,q}\to \Tilde{E}^\bullet_{p,q}$, and since $\tau$ is a homology isomorphism in degrees $i\leq \alpha$, the map between the $E^2$ pages
    
    \[\begin{tikzcd}
            {E^2_{**}=H_*\left((\Map_{\GL_d}(\Fr(TW|_L),\Theta)\parallelsum G)_0 \; ;\; H_*(\mmm{M}^\Theta(W_N,\rho_{W_N}))\right)} \arrow[d]
            \\
            {\tilde{E}^2_{**}=H_*\left((\Map_{\GL_d}(\Fr(TW|_L),\Theta)\parallelsum G)_0 \; ;\; H_*(\mmm{M}^\Theta(W,\rho_{W}))\right)} 
        \end{tikzcd}\]
        is an isomorphism for all $q\leq \alpha$. Then by the Spectral Sequence Argument recalled in Lemma~\ref{specseqargument}, $D$ induces an isomorphism
    \[\begin{tikzcd}
            {H_i(\mmm{M}^{\Theta}_G(W,\rho_W)}) \arrow[r, "\cong"] &
            {H_i(\mmm{M}^\Theta(W,\rho_{W})\times (\Map_{\GL_d}(\Fr(TW|_L),\Theta)\parallelsum G)_0}
    \end{tikzcd}\]
    for all $i\leq \alpha$.
\end{proof}

We give an application of this theorem for high even-dimensional manifolds using \cite[Corollary 1.7]{MR3665002}, which we recalled in Theorem \ref{thm:hom-st-high-dim}.

\begin{corollary}\label{cor:decoupling-sub-high-dim}
    Let $(W,L)$ be an $L$-decorated manifold, with $W$ a compact simply-connected manifold of dimension $2n\geq6$, and $L$ of dimension less than $n$. Let $\rho_W$ be an $n$-connected $\Theta$-structure on $W$, and denote by $g$ the stable genus of $W$. Then for all $i\leq\hdrange$ and $G\subset \ima e_{L}$, the decoupling map $D_L$ induces an isomorphism
    \[H_i(\mmm{M}^\Theta_G(W,\rho_W))\cong H_i\left(\mmm{M}^\Theta(W,\rho_W)\times (\Map_{\GL_d}(\Fr(TW|_L),\Theta)\parallelsum G)_0\right)\]
    where $(-)_0$ is the path component of the image of $\rho_W$.
\end{corollary}

\begin{proof}
    We know that $N$ is homotopy equivalent to $L$, and that the boundary of $N$ is a sphere bundle over $L$ with fibre $S^{c-1}$, where $c$ is the codimension of $L$ and $W$. The dimension assumption on $L$ implies that $c\leq n+1$, and therefore the pair $(N,\partial N)$ is $(n-1)$-connected. Moreover, by Lemma~\ref{lemma:genus}, the stable genus of $W_N$ is equal to $g$. Hence we are under the hypothesis of Theorem \ref{thm:hom-st-high-dim} and
    \[\tau:H_i(\mmm{M}^\Theta(W_N,\rho_{W_N}))\to H_i(\mmm{M}^\Theta(W,\rho_W))\]
    is an isomorphism for all $i\leq\hdrange$. By Theorem \ref{decouplingL}, the result follows.
\end{proof}

\medskip
\subsection{Decoupling unlinked circles}

In this section, we apply Theorem \ref{decouplingL} to the specific case where $L$ is a collection of $k$ unlinked circles. Therefore, throughout this section, we assume $W$ to be a compact simply-connected manifold of dimension $2n\geq6$, to satisfy the hypothesis of Corollary \ref{cor:decoupling-sub-high-dim}. Note that, since $W$ is simply-connected, it is always orientable.

\begin{definition}
    An embedding $f:\dcup{k}S^1\to W\setminus\partial W$ is said to be \emph{unlinked} if it extends to an embedding $\overline{f}:\dcup{k}D^2\to W\setminus\partial W$. If $W$ is oriented and $2$-dimensional we also assume that the embedding $\overline{f}$ is orientation preserving.
\end{definition}
    
\begin{notation}
    Throughout this section, we let $kS^1$ denote the space $\dcup{k}S^1$, and $kD^2$ denote the space $\dcup{k}D^2$.
\end{notation}

In this section, we will repeatedly use the following result, which follows from \cite[Chapter 8, Theorems 3.1, 3.2]{Hirsch}.

\begin{lemma}\label{lemma:swapping discs}
    Let $W$ be a connected $d$-manifold and and $f,g:kD^2\hookrightarrow W$ embeddings of $k$ disjoint discs into $W$. If $d=2$ and $W$ is oriented, assume also that $f$ and $g$ both preserve, or both reverse, orientation. Then there is a diffeomorphism $\phi$ of $W$ which is diffeotopic to the identity, such that $\phi\circ f=g$.
\end{lemma}

An immediate consequence of the result above is the following

\begin{lemma}\label{lemma: diffkS1 independednt of choice}
    The isomorphism type of $\diff_{f(kS^1)}(W)$ does not depend on the choice of the unlinked embedding $f:kS^1\hookrightarrow W$.
\end{lemma}

\begin{proof}
    For any two unlinked embeddings $f,g:kS^1\hookrightarrow W$, there are embeddings $\overline{f},\overline{g}:kD^2\hookrightarrow W$ extending $f,g$. By Lemma~\ref{lemma:swapping discs}, there is a diffeomorphism $\phi$ of $W$ with $\phi\circ f=g$. Then conjugation with $\phi$ defines an isomorphism between the group of diffeomorphism preserving $f(kS^1)$ and the one preserving $g(kS^1)$.
\end{proof}

From here on, we denote by $\diff_{kS^1}$ the isomorphism type of $\diff_{f(kS^1)}(W)$ for any embedding $f:kS^1\hookrightarrow W$, which is well-defined by Lemma~\ref{lemma: diffkS1 independednt of choice}.

We want to use Theorem \ref{decouplingL} for the case where the submanifold $L$ is an collection of unlinked circles, but instead of choosing a subgroup of $G$, we will take $G=\ima e_{kS^1}$, so we start by analysing what this image is. Fix an unlinked embedding of $kS^1$ in $W$ (we will refer to it as $kS^1\subset W$), since $W$ is orientable and by fixing a Riemmannian metric we get an explicit isomorphism
    \[TW_{|kS^1}\cong TkS^1\oplus \nu_{kS^1}\]
    
\begin{lemma}\label{lemma:isos of the normal bundle}
    There is a quotient map $q:\iso(TW_{|kS^1},TkS^1)\to \iso(\nu_{kS^1})$ which is a homomorphism, a Serre fibration and a homotopy equivalence.
\end{lemma}

\begin{proof}
     Any isomorphism $\overline{f}\in\iso(TW_{|kS^1},T{kS^1})$ satisfies $\overline{f}(TkS^1)=TkS^1$. Therefore, it induces a map on the quotient bundle $[\overline{f}]:TW_{|kS^1}/TkS^1=\nu_{kS^1}\to \nu_{kS^1}$. Using the inclusion $\nu_{kS^1}\to TW_{|kS^1}$ induced by the choice of a Riemannian metric, it is simple to check that this map satisfies the homotopy lifting property of Serre fibrations. Moreover, using the identification $TW_{|kS^1}\cong TkS^1\oplus \nu_{kS^1}$, we can verify easily that the fibre of $q$ over the identity is the space of sections of the vector bundle $\Hom(\nu_{kS^1}, T{kS^1})\to {kS^1}$ which is contractible.
\end{proof}

Since the normal bundle of the marked circles is also orientable and any orientable vector bundle over a circle is trivial, we know there is a bundle isomorphism $\nu_{kS^1}\cong kS^1\times \mm{R}^{d-1}$ giving a short exact sequence 
    \[\begin{tikzcd}
        C^\infty(S^1,\GL_{d-1})^k \ar[r] & \iso(\nu{kS^1}) \ar[r, "f"] & \diff(kS^1)
    \end{tikzcd}\]
where $f$ takes an isomorphism of $\nu_{kS^1}$ to the underlying diffeomorphism of the base $kS^1$. The map $\diff(kS^1)\to \iso(\nu{kS^1})$ defined by taking $\phi$ to the isomorphism $\phi\times {\Id}$ is a section for $f$, and therefore
    \begin{equation}\label{eq:map-eks1}
    \iso(\nu_{kS^1})\cong C^\infty(S^1,\GL_{d-1})^k\rtimes\diff(kS^1).
    \end{equation}
Fixing such isomorphism, the evaluation map \eqref{evaluation-normal-bundle} together with the quotient map of Lemma \ref{lemma:isos of the normal bundle} induce a homomorphism
    \begin{equation*}
        \overline{e_{kS^1}}:\begin{tikzcd}\diff_{kS^1}(W) \arrow[r] & C^\infty(S^1,\GL_{d-1})^k\rtimes\diff(kS^1).
        \end{tikzcd}
    \end{equation*}
We want to determine the image of the map $\overline{e_{kS^1}}$, which is equivalent to identifying the isomorphisms of the normal bundle of the circles that can actually be realised by a diffeomorphism of $W$. 

\begin{figure}[h]
        \hfill
        \subfigure{\includegraphics[width=0.45\textwidth]{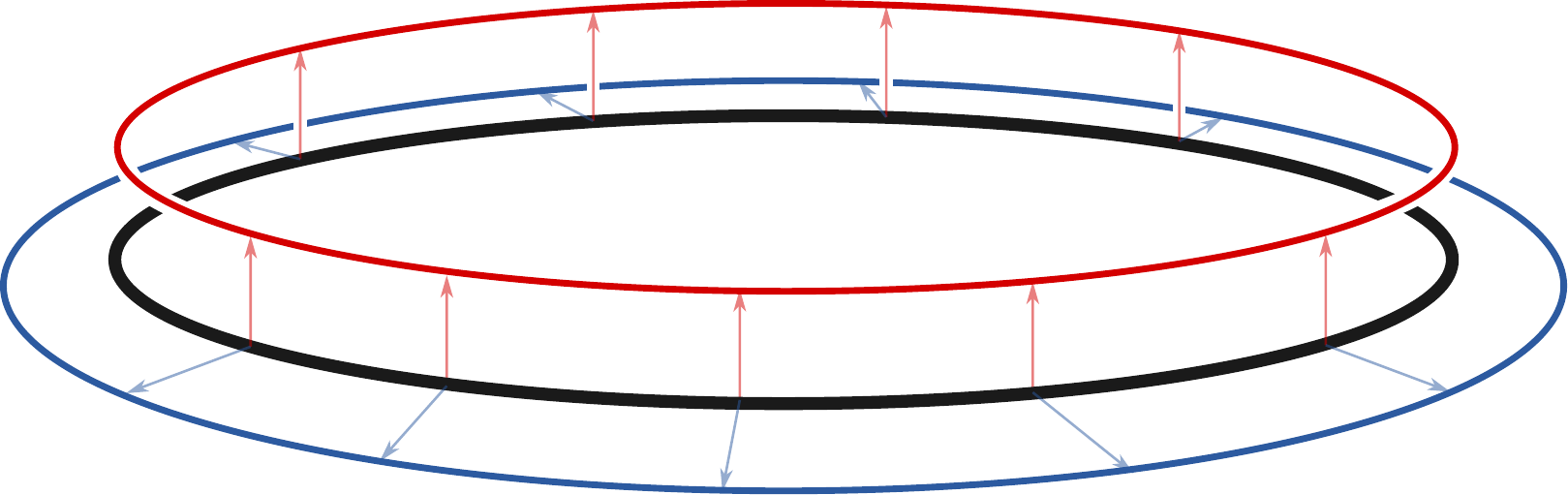}}
        \hfill
        \subfigure{\includegraphics[width=0.45\textwidth]{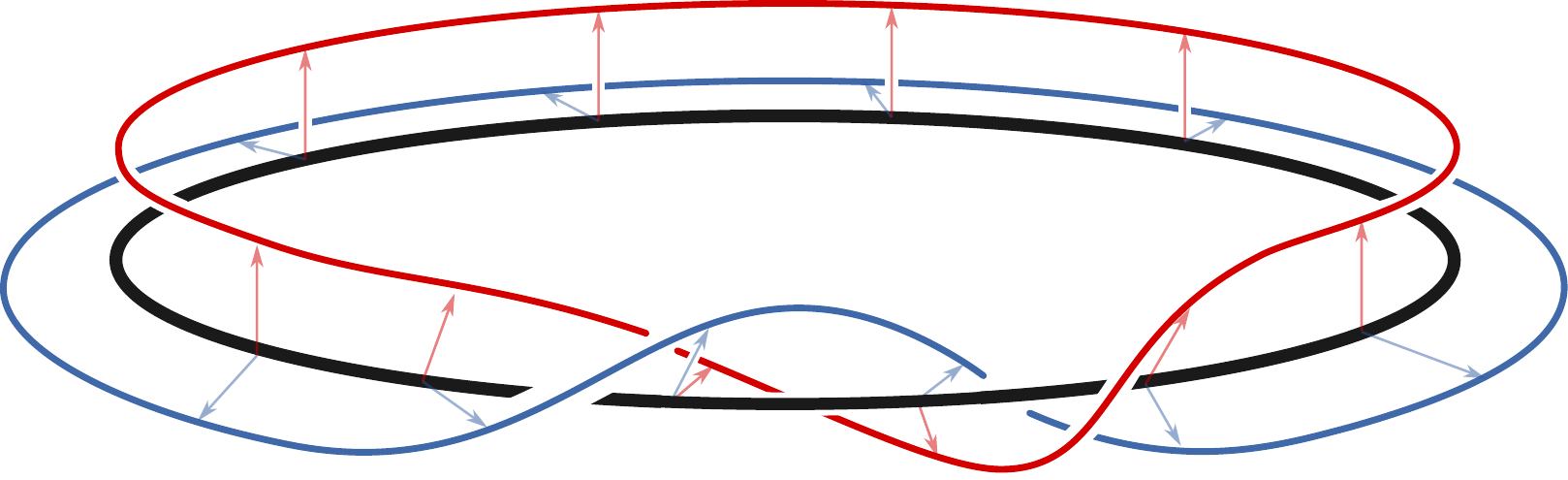}}
        \hfill \hfill
        \caption{A non trivial isomorphism of the normal bundle of $S^1$ in a $3$-dimensional manifold.}
        \label{fig:twist}
\end{figure}

For simplicity, we first look at the surface case:

\begin{lemma}\label{lemma: surjectivity for surfaces}
    Let $S_{g,b}$ be the oriented surface of genus $g$ and $b\geq 1$ boundary components. Then the image of $\overline{e_{kS^1}}$ is
    \[C^\infty(S^1,\GL_1^+)\rtimes\diff^+(kS^1).\]
\end{lemma}

\begin{proof}
     Fix once and for all an embedding $f:kD^2\hookrightarrow S_{g,b}$, and consider $f(\partial kD^2)$ to be the $kS^1$ decoration in $S_{g,b}$. We start by showing that any diffeomorphism of $\diff^+(kS^1)$ can be realised by an element in $\diff_{kS^1}(S_{g,b})$. Any diffeomorphism $\phi\in\diff^+(kS^1)$ can be extended to $\overline{\phi}\in\diff^+(kD^2)$ \cite[Chapter 8, Theorem 3.3]{Hirsch}, so it is sufficient to find a diffeomorphism $\psi$ of $W$ such that $\psi\circ f=f\circ \overline{\phi}$. But this can always be done, by Lemma~\ref{lemma:swapping discs}.
     
     This implies that the coset $\overline{e_{kS^1}}(Id) \cdot \diff^+(kS^1)$ is contained in the image of $\overline{e_{kS^1}}$, and since this map is surjective on path components, we know that $C^\infty(S^1,GL_{1}^+)^k\rtimes \diff^+(kS^1)$ is contained in the image of $\overline{e_{kS^1}}$. Moreover, since we assume $b\geq 1$, then for any $k$, all elements of $\diff_{kS^1}(S_{g,b})$ are orientation preserving and fix $f(kD^2)$ as a set, which means that $\ima \overline{e_{kS^1}}$ is contained in $C^\infty(S^1,GL_{d-1}^+)^k\rtimes \diff^+(kS^1)$, as required.
\end{proof}

We now analyse $\ima \overline{e_{kS^1}}$ for a manifold $W$ of higher dimension. In particular, we need to understand when there exists a diffeomorphism of $W$ that induces a loop with non-trivial homotopy class in $\pi_1(\GL_{d-1})$ as depicted in Figure \ref{fig:twist}. It is clear that determining this image does not only depend on the orientability of $W$, as it was for the case of points and discs, but it will also depend on the spinnability of $W$.

\begin{lemma}\label{surjectivity}
    Let $W$ be a simply-connected manifold of dimension $d\geq 5$. Then the image of $\overline{e_{kS^1}}$ is 
        \[C^\infty_\diamond(S^1,\GL_{d-1})^k\rtimes \diff(kS^1)\]
    where $C^\infty_\diamond(S^1,-)$ is equal to the subspace $C^\infty_{null}(S^1,-)$ of nullhomotopic loops if $W$ is spinnable, and is equal to $C^\infty(S^1,-)$ otherwise. 
\end{lemma}

Here we are not thinking of the spaces as pointed, so we consider a nullhomotopic loop to be one that is homotopic to a constant loop, not necessarily at a base point.

\begin{proof}
    Fix once and for all an embedding $f:kD^2\hookrightarrow W$, and consider $f(\partial kD^2)$ to be the $kS^1$ decoration in $W$. By the same argument as in the proof of Lemma~\ref{lemma: surjectivity for surfaces}, we know that any diffeomorphism of $\diff(kS^1)$ can be realised by an element in $\diff_{kS^1}(W)$. Without loss of generality, we assume from now on that $k=1$.
    
    Since $\overline{e_{S^1}}$ is surjective on path components, we know that $C^\infty_{null}(S^1,GL_{d-1}^+)\rtimes \diff(S^1)$ is contained in the image of $\overline{e_{S^1}}$ because it is the path component of $\overline{e_{S^1}}(Id)$.
    
    By Lemma~\ref{lemma:swapping discs} there exists an orientation preserving diffeomorphism $\phi_c\in \diff_{S^1}(W)$ that restricts to complex conjugation along the marked circles. This implies that $\phi_c$ induces an orientation reversing diffeomorphism on the normal bundle of $S^1$. Since the marked circle bounds an embedded $2$-disc, we can define such a $\phi_c$ by taking an embedded disc $D^d$ in $W$ containing the marked circle in its equator, and applying a rotation that flips the circle. By the Isotopy Extension Theorem, such a rotation can be extended to an isotopy in $W$. Then $\overline{e_{S^1}}(\phi_c)$ is contained in $C^\infty_{null}(S^1,GL_{d-1})\rtimes \diff(S^1)$. Since $\overline{e_{S^1}}$ is surjective on path-components, we conclude that $$C^\infty_{null}(S^1,GL_{d-1})\rtimes \diff(S^1)\subset \ima \overline{e_{S^1}}.$$

    We now show that a smooth curve $\gamma \notin C^\infty_{null}(S^1,\GL_{d-1})$ is in the image of $\overline{e_{S^1}}$ if, and only if, $W$ is not spin.
    
    First, assume $W$ is spin, and choose $\phi\in\diff_{S^1}(W)$. We know that any diffeomorphism of the circle is isotopic either to the identity or to complex conjugation, so without loss of generality, we can assume that $\phi$ restricts to one of these two maps on the marked circle. Start by assuming that $\phi$ restricts to the identity on the marked circle. Then using the embedding $f:D^2\hookrightarrow W$ bounding the marked circle, we can define a continuous function $g:S^2\to W$ by sending the bottom hemisphere $D^2_-$ to $f(D^2)$, and the top hemisphere $D^2_+$ to $\phi\circ f(D^2)$. Since we assume $W$ to be spin, we know that $w_2(g^*(TW))=g^*(w_2(TW))=0$. Since $d\geq 5$, the class $w_2$ detects the only obstruction to lifting to to $E\SO(d)$ the map $S^2\to B\SO(d)$ classifying the bundle $g^*(TW)$. Since $w_2(g^*(TW))=0$, this implies that $g^*(TW)$ is a trivial bundle, and in particular, its clutching function $S^1\to \GL_d^+$ is nullhomotopic. But note that $D\phi$ along the marked circle is a clutching function of $g^*(TW)$, and therefore $\overline{e_{S^1}}(\phi)$ is contained in $C^\infty_{null}(S^1,\GL_{d-1}^+)\rtimes \diff(S^1)$.
    
    On the other hand, if $\phi$ restricts to complex conjugation on the marked circle, then composing with the map $\phi_c$ constructed above, we get a map restricting to the identity. By the same arguments as above, we can conclude that $\overline{e_{S^1}}(\phi)\in C^\infty_{null}(S^1,\GL_{d-1})\rtimes \diff(S^1)$.
    
    Now assume $W$ is not spin. Since $W$ is simply-connected, by Hurewicz theorem, all its second homology classes are represented by maps $S^2\to W$ and by \cite[Theorem II.27]{thom1954quelques} we can always pick a representative given by an embedding. Since $W$ is not spin, there exists an embedding $h:S^2\to W$ such that $w_2(h^*(TW))\neq0$, and since $d\geq 5$, we can pick one such $h$ not intersecting $f(D^2)$ by the Transversality Theorem (see \cite[Corollary 12.2.7]{kupers2019lectures}). By Lemma~\ref{lemma:swapping discs}, we know there exists a diffeomorphism $\phi_h$ of $W$ taking $\phi_h\circ h|_{D^2_-}$ to $h|_{D^2_+}$, and which restricts to the identity on the image of $f(D^2)$. By definition, $D\phi_h|_{h(S^1)}$ is a clutching function for $h^*(TW)$ and therefore is not nullhomotopic as $h^*(TW)$ is non-trivial. 
    
    Let $\psi$ be a diffeomorphism taking $f(D^2)$ to $h|_{D^2_-}$. Then $\psi^{-1}\circ \phi_h\circ \psi$ is a diffeomorphism of $W$ whose image through $\overline{e_{S^1}}$ is not in $C^\infty_{null}(S^1,\GL_{d-1}^+)\rtimes\diff(S^1)$. Since $\overline{e_{S^1}}$ is surjective on path components, we conclude that the image of $\overline{e_{S^1}}$ contains $C^\infty(S^1,\GL_{d-1}^+)\rtimes \diff(S^1)$. Analogously, looking at the composition $\psi^{-1}\circ \phi_f\circ \psi\circ \phi_c$, we conclude that the image of $\overline{e_{S^1}}$ is $C^\infty(S^1,\GL_{d-1})\rtimes \diff(S^1)$.
\end{proof}

\begin{remark}
    Lemma~\ref{surjectivity} can be generalised for dimension $4$ assuming the manifold is spin, using exactly the same argument as above. 
\end{remark}

Now that we have analysed the image of $\overline{e_{kS^1}}$ we apply Theorem \ref{decouplingL}. As before, we start by looking at the surface case.

\begin{corollary}\label{cor:decouping-circles-in-surfaces}
    Let $S_{g,b}$ be the oriented surface of genus $g$ and $b\geq 1$ boundary components. Then for all $3i\leq 2g$
    \[H_i(B\diff_{kS^1}(S_{g,b}))\cong H_i(B\diff(S_{g,b})\times B(\Sigma_k\wr\SO(2))).\]
\end{corollary}

\begin{proof}
     Start by fixing an unlinked embedding $f:kS^1\hookrightarrow S_{g,b}$ and $\overline{f}:kD^2\hookrightarrow S_{g,b}$ extending $f$. The proof follows from applying Theorem \ref{decouplingL} taking $L$ to be the embedded circles. To do this, we start by verifying that the hypothesis of the Theorem are satisfied, ie. that the map defined by extending diffeomorphisms by the identity
     \[\varepsilon_N:\diff(S_{g,b}\setminus N)\to \diff(S_{g,b})\]
     induces isomorphisms on the homology groups of the classifying spaces in the range $3i\leq 2g$.
     
     Since the embedded circles are unlinked, we know they are nullhomotopic and therefore $S_{g,b}\setminus N$ is diffeomorphic to $S_{g,b+k}\cup\dcup{k}D^2$. Then 
     \[\diff(S_{g,b}\setminus N)\cong \diff(S_{g,b+k})\times\diff(D^2)^k.\]
     
     The inclusion $S_{g,b}\setminus N\to S_{g,b}$ induces maps 
        \[\varepsilon_S:\diff(S_{g,b+k})\to \diff(S_{g,b})\]
        \[\varepsilon_D:\diff(D^2)^k\to \diff(S_{g,b})\]
    given by extending the the diffeomorphisms by the identity, and it is simple to check that the following diagram commutes
    \[\begin{tikzcd}
        & \diff(S_{g,b})\times\diff(D^2)^k \ar[dr, "c"] &\\
        \diff(S_{g,b+k})\times\diff(D^2)^k \ar[ur, "\varepsilon_S\times id"] \ar[rr, "\varepsilon_N"] & & \diff(S_{g,b})
    \end{tikzcd}\]
    where $c(\phi,\psi):=\phi\circ \varepsilon_D(\psi)$.
    
    By \cite[Theorem 7.1]{randal2016resolutions}, we know that the map $\varepsilon_S$ induces a homology isomorphism in the range $3i\leq 2g$. Moreover, the map $c$ is a homotopy equivalence since the space $\diff(D^2)$ is contractible. Therefore, $\varepsilon_N$ induces a homology isomorphism in the range $3i\leq 2g$.
    
    Hence we are under the hypothesis of Theorem \ref{decouplingL}. Moreover, we know that $TW|_{kS^1}$ is trivial so 
        \[\Map_{\GL_d}(\Fr(TW|_{kS^1}),\Theta^{or})\simeq \Map(S^1,\{\pm 1\})^k\]
    is simply a disjoint union of points. And, by Lemma \ref{lemma: surjectivity for surfaces}, 
        \[\ima e_{kS^1}\simeq \ima \overline{e_{kS^1}}\cong C^\infty(S^1,\GL_1^+)^k\rtimes\diff^+(kS^1)\simeq \diff^+(kS^1)\simeq \Sigma_k\wr\SO(2).\]
    Applying Theorem \ref{decouplingL}, the result follows.
\end{proof}

Corollary \ref{cor:decouping-circles-in-surfaces} can be re-stated with a geometric interpretation. As discussed in Section \ref{sec:geometric-view}, the space  $\Emb(S_{g,b},\mm{R}^\infty)$ is a model for $E\diff(S_{g,b})$, and therefore it is also a model for $E\diff_{kS^1}(S_{g,b})$. With this model, the elements of $B\diff_{kS^1}(S_{g,b})$ are oriented submanifolds of $\mm{R}^\infty$ diffeomorphic to $S_{g,b}$ with $k$ marked unlinked circles. With this model, the forgetful map 
    \begin{equation}
        F_{kS^1}:B\diff_{kS^1}(S_{g,b})\to B\diff(S_{g,b})
    \end{equation}
simply forgets the marked circles. 

To interpret the evaluation map $E_{kS^1}$ in this model, we recall a definition that will also be useful for the interpretation of the evaluation map for the moduli space in higher dimensions with general tangential structures.

\begin{definition}\label{def: config of circles}
    Let $W$ be a manifold and $X$ be a space with an action of $\diff(S^1)$, the \emph{space of $k$-unlinked circles in $W$} with labels in $X$ is defined to be
        \[\confks(W;X):=\Emb^{\mathrm{unl}}(kS^1,W)\times X^k/\diff(kS^1)\]
    where $\Emb^{\mathrm{unl}}$ denotes the space of unlinked embeddings.
\end{definition}

Note that if $W$ is a simply connected manifold of dimension $d\geq 5$, all embeddings of $kS^1$ into $W$ are unlinked. 

Then a model for $B\diff^+(kS^1)$ is precisely the configuration space $\confks(\mm{R}^\infty;\{\pm 1\})$ of $k$ circles in $\mm{R}^\infty$ with labels in $\{\pm 1\}$, and the evaluation map $E_{kS^1}$ simply takes an oriented decorated submanifold $S$ in $\mm{R}^\infty$, to the configurations given by the marked $k$ oriented circles. 

\begin{corollary}\label{cor:decouping-circles-in-surfaces-conf}
    Let $S_{g,b}$ be the oriented surface of genus $g$ and $b\geq 1$ boundary components. Then for all $3i\leq 2g$
    \[H_i(B\diff_{kS^1}(S_{g,b}))\cong H_i(B\diff(S_{g,b})\times \confks(\mm{R}^\infty;\{\pm 1\})).\]
\end{corollary}

Analogous results for other tangential structures can be obtained by the same arguments. We now look at how the result above generalises for higher dimensions. From now on, we let $L(-)\coloneqq \Map(S^1,-)$ denote the free loop space.

\begin{corollary}\label{thm:decoupling-circles}
    Let $W$ be a compact simply-connected manifold of dimension $2n\geq6$, $\rho_W$ an $n$-connected $\Theta$-structure on $W$, and denote by $g$ the stable genus of $W$. Then for all $i\leq\hdrange$, the decoupling map $D_{kS^1}$ induces an isomorphism 
    \[H_i(\mmm{M}^\Theta_{kS^1}(W,\rho_W))\cong H_i\left(\mmm{M}^\Theta(W,\rho_W)\times \circlepathcomponent\right)\]
    where $(-)_0$ is the path component of the image of $\rho_W$, $L_\diamond(-)$ is equal to the subspace $L_{null}(-)$ of nullhomotopic loops if $W$ is spinnable, and is equal to $L(-)$ otherwise.
\end{corollary}

\begin{proof}
     The result follows from applying Corollary \ref{cor:decoupling-sub-high-dim}, taking $G=\ima e_{kS^1}\simeq \ima \overline{e_{kS^1}}$, which was identified in Lemma \ref{surjectivity}. Moreover, since $TkS^1$ is orientable, it is a trivial bundle and therefore the space $\Map_{\GL_d}(\Fr(TkS^1),\Theta)$ is equivalent to the space of continuous maps $kS^1\to \Theta$, which is precisely $L(\Theta)^k$.
     
     Then, by Corollary \ref{cor:decoupling-sub-high-dim}, for all $i\leq\hdrange$, the decoupling map $D_{kS^1}$ induces an isomorphism
        \[H_i(\mmm{M}^\Theta_{kS^1}(W,\rho_W))\cong H_i\left(\mmm{M}^\Theta(W,\rho_W)\times (L(\Theta)^k\parallelsum C^\infty_\diamond(S^1,\GL_{d-1})^k\rtimes \diff(kS^1))_0\right)\]
     Moreover, the space $(L(\Theta)^k\parallelsum C^\infty_\diamond(S^1,\GL_{d-1})^k\rtimes \diff(kS^1))_0$ is homotopy equivalent to
        \begin{equation}\label{eq:base space for circles}
            (L(\Theta)\parallelsum C^\infty_\diamond(S^1,\GL_{d-1}))^k\parallelsum \diff(kS^1))_0
        \end{equation}
    Taking $\Emb(kS^1,\mm{R}^\infty)$ as the model for $E\diff(kS^1)$, we get a model for the space in \ref{eq:base space for circles}, which is precisely the configuration space of $k$ circles in $\mm{R}^\infty$ with labels in $(L(\Theta)\parallelsum C^\infty_\diamond(S^1,\GL_{d-1}))_0$, as required. Since the space of smooth loops is homotopy equivalent to the free loop space, the result follows.
\end{proof} 
\section{Decoupling for general tangential structures in higher dimensions}\label{section:general-tangential-structures}

\noindent In this section we show how Corollaries \ref{higherdimensions} and \ref{cor:decoupling-sub-high-dim} can be generalised for other tangential structures, based on the techniques used by Galatius and Randal-Williams in \cite[Section 9]{MR3665002}. Recall that the decoupling theorems (\ref{decoupling} and \ref{decouplingL}) relied on the hypothesis that the map 
$$\mmm{M}^\Theta(W_N,\rho_{W_N}) \rightarrow \mmm{M}^\Theta(W,\rho_{W}) $$
induces a homology isomorphism in a range. In even dimensions at least $6$, this assumption was shown to hold in several cases in \cite[Corollary 1.7]{MR3718454} as recalled in \ref{thm:hom-st-high-dim}, but only when the $\Theta$-structure $\rho_W:\Fr(TW)\rightarrow \Theta$ is $n$-connected (ie. the induced map $\pi_i(\Fr(TW))\rightarrow \pi_i(\Theta)$ is an isomorphism for $i<n$ and an epimorphism for $i=n$). In \cite[Section 9]{MR3718454} Galatius and Randal-Williams provide a generalisation of the result to general tangential structures. In this section we introduce the tools used to construct this generalisation and show how they also provide an extension of the decoupling result in higher dimensions for general tangential structures.

One could hope that for any manifold $W$ and any $\Theta$-structure $\rho_W$, the decoupling map would still induce a homology isomorphism, but this is not the case, as it is shown by the following example.

\begin{example}
    Consider the manifold $W_g=\#_gS^n\times S^n$ with one embedded disc as a decoration. Let $W_{g,1}=\#_g(S^n\times S^n)\setminus \int(D^{2n})$ and recall there is an isomorphism $$\diff^+(W_{g,1})\xrightarrow{\cong}\diff^+_1(W_g)$$
    given by extending the diffeomorphism of $W_{g,1}$ by the identity on the marked disc (see Lemma \ref{fibrationongroups}).
    
    Therefore, the decorated moduli space $\mmm{M}^{or}_1(W_g,\rho_W)\simeq B\diff^+_1(W_g)$ is weakly equivalent to $\mmm{M}^{or}(W_{g,1},\rho_W)\simeq B\diff^+(W_{g,1})$. In this case, the decoupling map 
    \begin{equation}\label{eq:not-an-iso}
        \begin{tikzcd}[row sep=small]
        \mmm{M}^{or}_1(W_g,\rho_{W_{g}})\ar[r, "D"]\ar[draw=none]{d}[sloped,auto=false]{\simeq} 
         & \mmm{M}^{or}(W_g,\rho_{W_{g}})\times \Theta^{or}_0 \ar[draw=none]{d}[sloped,auto=false]{\simeq}\\
        \mmm{M}^{or}(W_{g,1},\rho_{W_{g,1}}) & \mmm{M}^{or}(W_g,\rho_{W_{g}})
        \end{tikzcd}
    \end{equation}
    does not induce a homology isomorphism on integral coefficients in a stable range as was shown in \cite[Sections 5.1 and 5.2]{galatius2018moduli}. This implies that the decoupling as stated in Corollary \ref{higherdimensions} is not true for general tangential structures.
\end{example}

Let $W$ be a $2n$-dimensional manifold, $2n\geq 6$, with possibly non-empty boundary, and $\lambda_W$ a $\Lambda$-structure on $W$. If the map $\lambda_W:\Fr(TW)\to \Lambda$ is not $n$-connected we will use an ``intermediate'' tangential structure $\Theta$ which is better behaved. Precisely, let the following be the Moore-Postnikov $n$-stage of $\lambda_W$,
\begin{equation}\label{eq:Moore Postnikov n-stage}
    \begin{tikzcd}
     & \Theta \arrow[rd, "u"] &         \\
    \Fr(TW) \arrow[rr, "\lambda_W"] \arrow[ru, "\rho_W"] &   & \Lambda.
\end{tikzcd}
\end{equation}
This means that $\Theta$ is a $\GL_d$-space, $u$ is an $n$-co-connected (ie. the induced map $\pi_i(\Fr(TW))\rightarrow \pi_i(\Theta)$ is an isomorphism for $i>n$ and a monomorphism for $i=n$) equivariant fibration and $\rho_W$ an $n$-connected equivariant cofibration. Such a factorization always exist and it is unique up to homotopy equivalence.

Denote by $\rho_\partial$ and $\lambda_\partial$ the restriction of $\rho_W$ and $\lambda_W$ respectively to $\Fr(TW)|_{\partial W}$.  Any $\Theta$-structure on $W$ induces a $\Lambda$-structure by postcomposition with $u$, giving us a map
$$ \Bun^\Theta(W,\rho_W)\rightarrow \Bun^{\Lambda}(W,\lambda_W). $$

\begin{lemma}[\cite{MR3665002}, Lemma 9.4]
    If $W$ is a manifold equipped with a $\Lambda$-structure $\lambda_W$ and $\Fr(TW)\xrightarrow{\rho_W} \Theta \xrightarrow{u} \Lambda$ is a Moore-Postnikov $n$-stage of $\lambda_W$, then the stable genus $\overline{g}(W,\rho_W)$ is equal to $\overline{g}(W,\lambda_W)$.
\end{lemma}

We now define a topological monoid that is crucial to the comparison between the moduli spaces $\mmm{M}^\Lambda(W,\lambda_W)$ and $\mmm{M}^\Theta(W,\rho_W)$.

\begin{definition}
    If $W$ is a closed manifold, denote by $\hAut(u)$ the group-like topological monoid consisting of equivariant weak equivalences $\Theta\to \Theta$ over $u$, ie. $\GL_d$-equivariant maps $\Theta\to \Theta$ fitting into the following commutative diagram
    \[\begin{tikzcd}
        \Theta \arrow[rr, "\simeq"] \arrow[dr, "u"'] && \Theta. \arrow[dl, "u"]\\
        & \Lambda &
    \end{tikzcd}\]
    If $W$ has non-empty boundary, let $\rho_\partial$ be the restriction of $\rho_W$ to $\partial W$. Denote by $\hAut(u,\rho_\partial)$ the group-like topological monoid consisting of equivariant weak equivalences $\Theta\to \Theta$ over $u$ and under $\rho_\partial$
    \[\begin{tikzcd}[column sep=tiny]
        & \Fr(TW)|_{\partial W} \arrow[dl, "\rho_\partial"'] \arrow[dr, "\rho_\partial"] &\\
        \Theta \arrow[rr, "\simeq"] \arrow[dr, "u"'] && \Theta. \arrow[dl, "u"]\\
        & \Lambda &
    \end{tikzcd}\]
\end{definition}

The monoid $\hAut(u,\rho_\partial)$ acts on the space of $\Theta$-structures on $W$ by post-composition, and the following result shows that this action encodes precisely the relation between $\Theta$ and $\Lambda$-structures on $W$.

\begin{lemma}[\cite{MR3665002}, Lemma 9.2]
    In the context defined above, the map induced by postcomposition with $u$
    \[\Bun^\Theta_{\rho_\partial}(W)\parallelsum \hAut(u,\rho_\partial)\to \Bun^{\Lambda}_{\lambda_\partial}(W)\]
    is a homotopy equivalence onto the path components which it hits.
\end{lemma}

Let $\hAut(u,\rho_\partial)_{[W,\rho_W]}$ denote the components of $\hAut(u,\rho_\partial)$ that map $\Bun^\Theta(W,\rho_W)$ to itself. By the orbit-stabiliser theorem
    \[\Bun^\Theta(W,\rho_W)\parallelsum \hAut(u,\rho_\partial)_{[W,\rho_W]}\to \Bun^{\Lambda}(W,\lambda_W)\]
is also a homotopy equivalence onto the path components which it hits. Taking a further Borel construction with the groups $\diff(W)$, $\diff^k_m(W)$, $\diff_L(W)$, we get that the induced maps 
    \begin{align}\label{eq:haut-equivalences}
        \mmm{M}^\Theta(W,\rho_W)\parallelsum \hAut(u,\rho_\partial)_{[W,\rho_W]}& \rightarrow \mmm{M}^{\Lambda}(W,\lambda_W)\\
        \mmm{M}^{\Theta,k}_{m}(W,\rho_W)\parallelsum \hAut(u,\rho_\partial)_{[W,\rho_W]}& \rightarrow \mmm{M}^{\Lambda,k}_m(W,\lambda_W)\\
        \mmm{M}^{\Theta}_{L}(W,\rho_W)\parallelsum \hAut(u,\rho_\partial)_{[W,\rho_W]}& \rightarrow \mmm{M}^{\Lambda}_L(W,\lambda_W)
    \end{align}
are weak homotopy equivalences. Therefore, analysing  $\hAut(u,\rho_\partial)_{[W,\rho_W]}$ and applying Corollaries \ref{higherdimensions} and \ref{cor:decoupling-sub-high-dim} we get decoupling results for general $\Lambda$.

\begin{lemma}[\cite{galatius2018moduli}]\label{lemma:hAut}
    If $(W,\partial W)$ is $c$-connected for some $c\leq n-1$, then the monoid $\hAut(u,\rho_\partial)$ is a non-empty $(n-c-2)$-type. In particular, it is contractible if $(W,\partial W)$ is $(n-1)$-connected.
\end{lemma}

We now focus on applying these techniques to the manifold $W_{g,1}=\#_g S^n\times S^n\setminus D^{2n}$, for $n\geq 3$. 

\begin{proposition}\label{prop:decoupling Wg1}
    Let $W_{g,1}=\#_g S^n\times S^n\setminus D^{2n}$, for $n\geq 3$, and $\lambda_{W}$ a $\Lambda$-structure on $W$. Let $g$ denote the stable genus $\overline{g}(W,\lambda_W)$. For all $i\leq \frac{g-4}{3}$, the group $H_i(\mmm{M}^{\Lambda,k}_m(W_{g,1},\lambda_{W}))$ is isomorphic to
    \[ H_i(\mmm{M}^{\Lambda}(W_{g,1}, \lambda_{W})\times \Theta^m_0\parallelsum\Sigma_m\times (\Theta\parallelsum \GL_{2n}^+)^k_0\parallelsum\Sigma_k)\]
    where $(-)_0$ denotes the path-component of $E(\rho_W)$, for $\rho_W$ as in \eqref{eq:Moore Postnikov n-stage}.
\end{proposition}

Note that in the above proposition, the decorations on $\mmm{M}^{\Lambda,k}_m(W)$ get decoupled into components depending on $\Theta$, the tangential structure that appeared in the Moore-Postnikov $n$-stage factorisation of $\lambda_W$. This is quite different than what was obtained in Corollary~\ref{higherdimensions} as well as in the other decoupling theorems of sections \ref{section:proof-of-decoupling} and \ref{section:corollaries}, where the decoupled components corresponding to the marked points and discs depended on the original chosen tangential structure $\Lambda$.

\begin{proof}
    By Lemma~\ref{lemma:hAut}, we know that $\hAut(u,\rho_\partial)$ is contractible, and therefore 
        \begin{align*}
            \mmm{M}^{\Theta}(W_{g,1},\rho_W) &\simeq \mmm{M}^{\Lambda}(W_{g,1},\lambda_W)\\
            \mmm{M}^{\Theta,k}_m(W_{g,1},\rho_W) &\simeq \mmm{M}^{\Lambda,k}_m(W_{g,1},\lambda_W)
        \end{align*}
    are weak homotopy equivalences.
    Since $\rho_W$ is $n$-connected, we can apply Corollary~\ref{higherdimensions} to understand the homology of $\mmm{M}^{\Theta,k}_m(W_{g,1},\rho_W)$. Putting this together with the above identifications we get that the group $H_i(\mmm{M}^{\Lambda,k}_m(W_{g,1},\lambda_W))$ is isomorphic to $i$th homology group of
    \begin{equation*}\label{eq:borel construction for general tg structures}
        \mmm{M}^{\Lambda}(W_{g,1}, \lambda_{W})\times \Theta^m_0\parallelsum\Sigma_m\times (\Theta\parallelsum \GL_{2n}^+)^k_0\parallelsum\Sigma_k.
    \end{equation*}
\end{proof}

We now look at the case where $\Lambda$ is the tangential structure for orientations: a $\GL_d$-equivariant map $\lambda_W:\Fr(TW_{g,1})\to \{\pm 1\}$ determines, up to a contractible choice, a map $\ell'_W:W_{g,1}\to B\SO(2n)$ fitting into the following homotopy pullback square
    \[\begin{tikzcd}
    \Fr(TW_{g,1}) \ar[r, "\lambda_W"] \ar[d]  \arrow[dr, phantom, "\lrcorner", very near start]& \{\pm 1\} \ar[d]\\
    W_{g,1} \ar[r, "\ell'_W"'] & B\SO(2n).
    \end{tikzcd}\]
Then an equivariant Moore-Postnikov factorization of $\lambda_W$ can be obtained from a Moore-Postnikov factorization of $\ell'_W$. Since $W_{g,1}$ is $(n-1)$-connected and parallelizable, we know that the $n$-stage of this factorization is given by maps
    \[\begin{tikzcd}
        W_{g,1} \ar[r, "\ell_W"] & B\O(2n)\langle n\rangle \ar[r, "u"] & B\SO(2n)
    \end{tikzcd}\]
where $B\O(2n)\langle n\rangle$ is the $n$-connected cover of $B\O(2n)$. Taking the pullback of $\{\pm 1\}\to B\SO(2n)$ along these maps, we get
    \[\begin{tikzcd}
    \Fr(TW_{g,1}) \ar[r] \ar[d]  \arrow[dr, phantom, "\lrcorner", very near start]& \O[0,n-1] \ar[r] \ar[d]  \arrow[dr, phantom, "\lrcorner", very near start]&  \{\pm 1\} \ar[d]\\
    W_{g,1} \ar[r, "\ell_W"'] & B\O(2n)\langle n\rangle \ar[r, "u"] & B\SO(2n)
    \end{tikzcd}\]
where $\O[0,n-1]$ is the $(n-1)$-truncation of $\O$. Note that a path-component of  $\O[0,n-1]$ is homotopy equivalent to $\SO[0,n-1]$, the $(n-1)$-truncation of $\SO$.

\begin{corollary}\label{cor:Wg1-orientation}
    Let $W_{g,1}=\#_g S^n\times S^n\setminus D^{2n}$, for $n\geq 3$. Then for all $i\leq \frac{g-4}{3}$, the group $H_i(B\diff^{+,k}_m(W_{g,1}))$ is isomorphic to 
    \[ H_i(B\diff^+(W_{g,1})\times \SO[0,n-1]^m\parallelsum\Sigma_m\times B\O(2n)\langle n\rangle^k\parallelsum\Sigma_k).\]
\end{corollary}

The proof is a direct application of Proposition \ref{prop:decoupling Wg1} using the factorization described above, and the fact that for an orientation $\rho_{W_{g,1}}:\Fr(TW_{g,1})\to\{\pm1\}$, the stable genus $\overline{g}(W_{g,1},\rho_{W_{g,1}})$ is equal to $g$ (see \cite[Section 3.2]{galatius2018moduli}).

We end by using the result above to explicitly compute the cohomology of $B\diff^{+,k}_m(W_{g,1})$ with rational coefficients, in the stable range. As an immediate consequence of Corollary~\ref{cor:Wg1-orientation} and Kunneth Theorem, the elements of $H^*(B\diff^{+,k}_m(W_{g,1});\mm{Q})$ of degree $i\leq \frac{g-4}{3}$, are given by the elements of such degrees in the tensor product of the cohomology rings of $B\diff^+(W_{g,1})$, $\SO[0,n-1]^m\parallelsum\Sigma_m$ and $B\O(2n)\langle n\rangle^k\parallelsum\Sigma_k$. 

By \cite[Corollary 1.8]{MR3718454}, in degrees $i\leq \frac{g-3}{2}$, the ring $H^*(B\diff^{+}(W_{g,1});\mm{Q})$ is isomorphic to
    \[\mm{Q}[\kappa_c|c\in\mmm{B},|c|>2n]\]
where $\mmm{B}$ denotes the set of monomials in the classes $e$, $p_{n-1}$, $p_{n-2},\dots, p_{\lceil\frac{n+1}{4} \rceil}$ of $H^*(B\SO(2n))$ and $|\kappa_c|=|c|-2n$.

By the Cartan-Leray spectral sequence, we also know that 
    \[H^*(B\O(2n)\langle n\rangle^k\parallelsum\Sigma_k;\mm{Q})\cong H^*(B\O(2n)\langle n\rangle^k;\mm{Q})^{\Sigma_k}\]
the fixed points by the action of $\Sigma_k$ which permutes the factors of $(B\O(2n)\langle n\rangle)^k$. We know that $H^*(B\O(2n)\langle n\rangle;\mm{Q})$ is simply the subalgebra of $H^*(B\O(2n);\mm{Q})=\mm{Q}[p_1,\dots,p_{n-1},e]$ with no generators of degrees $\leq n$. So 
    \[H^*(B\O(2n)\langle n\rangle;\mm{Q})\cong \mm{Q}[e, p_{\lceil\frac{n+1}{4}\rceil},\dots,p_{n-1}].\]
Then $H^*(B\O(2n)\langle n\rangle^k;\mm{Q})^{\Sigma_k}$ is isomorphic to
    \[\left(\bigotimes\limits_k\mm{Q}[e, p_{\lceil\frac{n+1}{4}\rceil},\dots,p_{n-1}]\right)^{\Sigma_k}\]
the fixed points by the action of $\Sigma_k$ which permutes the factors of the $k$-fold tensor product.

Analogously,
    \[H^*(\SO[0,n-1]^m\parallelsum\Sigma_m;\mm{Q})\cong H^*(\SO[0,n-1]^m;\mm{Q})^{\Sigma_m}.\]
Using the fibre sequence 
    \[\begin{tikzcd} \SO(2n)\langle n-1\rangle \ar[r] & \SO(2n) \ar[r] & \SO[0,n-1] \end{tikzcd}\]
then by the Leray-Hirsch Theorem, we have a $\mm{Q}$-module isomorphism
    \[H^*(\SO(2n);\mm{Q})\cong H^*(\SO[0,n-1];\mm{Q})\otimes H^*(\SO(2n)\langle n-1\rangle;\mm{Q}).\]
Together with the fact that we have a canonical ring monomorphism
    \[\begin{tikzcd} H^*(\SO(2n)\langle n-1\rangle;\mm{Q})  \ar[r] & H^*(\SO(2n);\mm{Q}) \end{tikzcd}\]
we conclude that 
    \[H^*(\SO[0,n-1];\mm{Q})\cong \bigwedge [y_1,\dots, y_{\lfloor\frac{n-1}{4} \rfloor}]\]
with $|y_i|=4i-1$.

Then $H^*(\SO(2n)[0,n-1]^m;\mm{Q})^{\Sigma_m}$ is isomorphic to
    \[\left(\bigotimes\limits_m\bigwedge [y_1,\dots, y_{\lfloor\frac{n-1}{4} \rfloor}]\right)^{\Sigma_m}\]
the fixed points by the action of $\Sigma_m$ which permutes the factors of the $m$-fold tensor product.

Therefore in degrees $i\leq \frac{g-4}{3}$, the ring $H^*(B\diff^{+,k}_m(W_{g,1}))$ is isomorphic to the graded commutative algebra
    \begin{align*}
        \mm{Q}[\kappa_c|c\in\mmm{B},|c|>2n]  \otimes\left(\bigotimes\limits_m\bigwedge [y_1,\dots, y_{\lfloor\frac{n-1}{4} \rfloor}]\right)^{\Sigma_m} \otimes \left(\bigotimes\limits_k\mm{Q}[e, p_{\lceil\frac{n+1}{4}\rceil},\dots,p_{n-1}]\right)^{\Sigma_k}.
  \end{align*}


\bibliographystyle{amsalpha}
\bibliography{bibliography}

\end{document}